\documentclass[11pt,twoside]{scrartcl}          

\usepackage{todonotes}

\usepackage{xifthen}
\usepackage{tikz}
\usepackage{tikz-3dplot}
\usepackage{tikz-3dplot-circleofsphere}
\usepackage{pgfplots}
\usepackage[latin1]{inputenc}
\usepackage[T1]{fontenc}
\usepackage{amsmath, amssymb, amsthm}    
\usepackage{mathrsfs}        
\usepackage{amstext, amsfonts, a4}
\usepackage{bbm}
\usepackage{wasysym}
\usepackage{lmodern}

 \usepackage[colorlinks=true]{hyperref}
\hypersetup{urlcolor=blue, citecolor=blue, linkcolor=blue}

\usepackage{color,transparent}
\usepackage{graphics,graphicx,subfigure}

\usepackage{comment}
\usepackage{float} 
\theoremstyle{plain}
	\newtheorem{Theo}{Theorem}[section] 
	\newtheorem{Prop}{Proposition}[section]       
	\newtheorem{Lem}{Lemma}[section]            

        \newtheorem{TheoPrinc}{Theorem}     

	\newtheorem{Defi}{Definition}[section]
        
	\newtheorem{Nota}{Notation}[section]

	\newtheorem{Rema}{Remark}[section]

\newcommand{\bs}{\symbol{92}}
\def\ogg~{{\rm \og}}   

\def\NN{{\mathbb N}}    
\def\ZZ{{\mathbb Z}}     
\def\RR{{\mathbb R}}    

  \def\cG{{\mathcal G}}     \def\cH{{\mathcal H}}   \def\cC{{\mathcal C}}     \def\cD{{\mathcal D}}   \def\cP{{\mathcal P}}  \def\cE{{\mathcal E}}       \def\cL{{\mathcal L}} \def\cR{{\mathcal R}}    

\newcommand{\dx}{\partial_x}                    
\newcommand{\dy}{\partial_y}

\newcommand{\dt}{\partial_t}

\begin{document}
\title{Null controllability of the parabolic spherical Grushin equation\footnote{This work was supported by a grant from the "Fondation CFM pour la Recherche". It was also partially supported by the iCODE Institute, research project of the IDEXParis-Saclay.}}
\author{Cyprien Tamekue\footnote{Laboratoire des Signaux et Systèmes (L2S), cyprien.tamekue@l2s.centralesupelec.fr}\vspace{0.2cm}\\\small{Université Paris-Saclay, CNRS, CentraleSupélec, 91190, Gif-sur-Yvette, France}}
\maketitle
\begin{abstract}
	We investigate the null controllability property of the parabolic equation associated with the Grushin operator defined by the canonical almost-Riemannian structure on the 2-dimensional sphere $\mathbb S^2$. This is the natural generalization of the Grushin operator $\mathcal G = \partial_x^2 + x^2\partial_y^2$ on $\mathbb R^2$ to this curved setting, and presents a degeneracy at the equator of $\mathbb S^2$. 
	
	We prove that the null controllability is verified in large time when the control acts as a source term distributed on a subset $\overline{\omega} = \{ (x_1,x_2,x_3)\in \mathbb S^2\mid \alpha<|x_3|<\beta \}$ for some $0\le\alpha<\beta\le 1$. 
	More precisely, we show the existence of a positive time $T^{*}>0$ such that the system is null controllable from $\overline{\omega}$ in any time $T\ge T^{*}$, and that the minimal time of control from $\overline{\omega}$ satisfies $T_{\mbox{\scriptsize{min}}}\ge\log(1/\sqrt{1-\alpha^2})$. Here, the lower bound corresponds to the Agmon distance of $\overline{\omega}$ from the equator.
	
	These results are obtained by proving a suitable Carleman estimate by using unitary transformations and Hardy-Poincaré type inequalities to show the positive null-controllability result. The negative statement is proved by exploiting an appropriate family of spherical harmonics, which concentrates at the equator, to falsify the uniform observability inequality.
	
\end{abstract}


\tableofcontents

\bigskip

\section{Introduction}\label{s.introduction} 
During the last decade, there has been a lot of interest in studying the null controllability of degenerate parabolic equations, that is to say, parabolic equations whose principal symbol can vanish inside the domain.
In \cite{Beau}, it has been shown that null controllability for the Grushin equation, which is an example of a degenerate parabolic equation, requires a non-trivial positive time. This is in stark contrast with what happens for the usual heat equations (see, for instance, \cite{Coron,Fursikov,Imanuvilov,Lebeau}), which are null controllable in an arbitrarily short time. 

The study of properties of null controllability of parabolic, spherical Grushin equation is relevant since the latter is both degenerate and singular at different locations. So it is natural to expect that such a study would be a bit more subtle than that of dimension $2$. Before going further, let us start by recalling some well-known results going in our direction.

\subsection{The $2D$ parabolic Grushin equation}\label{ss.intro1}
Since its introduction in Baouendi \cite{Baouendi} (see also Grushin \cite{Grushin}), the so-called Grushin (or Grushin-Baouendi) operator, defined
in the bidimensional setting as
\begin{equation}\label{eq1}
\cG:=\dx^{2}+x^{2}\dy^{2},
\end{equation}
where $(x, y)\in \RR^{2}$, received considerable attention in the field of differential geometry, as well as in control theory as a prototypical example of a degenerate elliptic, hypoelliptic operator, see for instance \cite{Beau,Boscain,Prandi,Casarino,Grushin}.

More recently, in \cite{Beau} the authors investigated the properties of null controllability for the degenerate parabolic equation associated with \eqref{eq1}, showing that they exhibit a wider range of behaviours. In particular, null controllability may hold true or not depending on the geometry of the open set $\omega$ and the time horizon $T$. 
More precisely, the authors considered the following parabolic equation, which presents a degeneracy at $x=0$:
\begin{equation}\label{eq2}
\begin{cases}
\partial_tf-\cG f = u(t,x,y)\mathbbm{1}_{\omega_0}(x,y),&(t,x,y)\in(0,T)\times\Omega_0,\cr
f(t,x,y) = 0,&(t,x,y)\in(0,T)\times\partial\Omega_0,\cr
f(0,x,y) = f_0(x,y),&(x,y)\in\Omega_0,
\end{cases}
\end{equation}
where $T>0$, $\Omega_0 = (-1,1)\times(0,1)$,
$\;\omega_0\subset\Omega_0$ is an open subset, $f$ is the state, $u$ is the control function, $f_0$ is the initial datum. 
Then, we have the following, see \cite[Theorem 1]{Beau}.

\begin{Theo}[\cite{Beau}]
	Let $\omega = (a,b)\times(0,1)$, where $0<a<b\le1$. Then, we have
	\begin{equation*}
	T_{\mbox{\scriptsize{min}}}:=\inf\{T>0: \mbox{system (\ref{eq2}) is null controllable from\;$\omega_0$\; in time\;}T\} \ge \frac{a^2}{2}.
	\end{equation*}	
\end{Theo}

This result is to be interpreted in the following sense: there exists a positive time $T^{*}>0$ such that system (\ref{eq2}) is null controllable from $\omega_0$ in any time $T>T^{*}$ and that the minimal time $T_{\mbox{\scriptsize{min}}}$ required for the null controllability of system (\ref{eq2}) from $\omega_0$ satisfies $T_{\mbox{\scriptsize{min}}}\ge\frac{a^{2}}{2}$.

Following this line of investigation, in \cite[Theorem 1.4]{Beau2}, Beauchard, Dardé, and Ervedoza consider the more general operator $\cG_q=\dx^2+q(x)^2\dy^2$, where $q$ is a real function, satisfying for some $L_{\pm}>0$ the following:
\begin{equation}\label{eq3}
q(0) = 0,\hspace{1cm}q\in\cC^{3}([-L_{-},L_{+}]),\hspace{1cm}\inf\limits_{(-L_{-},L_{+})}q'>0.
\end{equation}
For the associated parabolic degenerate equation on $\Omega=(-L_-,L_+)\times (0,\pi)$, with boundary control at the vertical side $\Gamma_+=\{L_+\}\times(0,\pi)$, and with initial datum $f_0\in H^1_0(\Omega)$ the authors are able to obtain the sharp value of the minimal time:
\begin{equation}\label{eq4}
T_{\mbox{\scriptsize{min}}} = \frac{1}{q'(0)}\int_{0}^{L_{+}}q(s)ds.
\end{equation}
Note that the integral in (\ref{eq4}) can be seen as the Agmon's distance associated to potential $q$ between $\{L_{+}\}$ (related to the control support) and $\{0\}$ (related to the degeneracy location). See also \cite{Beau4} for preliminary results in this direction.

\subsection{Setting, main results and strategy of proofs}\label{ss.intro2}
Let us consider the $2-$dimensional sphere $\mathbb{S}^{2} = \{p = (x_{1},x_{2},x_{3})\in\RR^{3}:x_{1}^{2}+x_{2}^{2}+x_{3}^{2} = 1\}$. We let $X$, $Y$ and $Z$ be vector fields generating the counter clock-wise rotations around the $x_{1}$, $x_{2}$ and $x_{3}$ axes respectively, viz.
\begin{equation}\label{eq::vectorfields}
X = -x_{3}\partial_{x_2}+x_{2}\partial_{x_3},\hspace{1cm} Y = -x_{3}\partial_{x_1}+x_{1}\partial_{x_3},\hspace{1cm} Z = -x_{2}\partial_{x_1}+x_{1}\partial_{x_2}. 
\end{equation}
Here, we are using the identification of vector fields with derivations. These vector fields, usually known as the Killing vector fields on $\mathbb{S}^{2}$ span at each point $p$ of $\mathbb{S}^{2}$ the tangent space $T_{p}\mathbb{S}^{2}$.

Observe that $\{X,Y\}$ are linearly independent outside of the equator $\cE := \{x_3=0\}$. Nevertheless, since $\left[X,Y\right] = Z$, the system of vector fields $\left\{X,Y\right\}$ is bracket-generating and determines a sub-Riemannian structure on $\mathbb{S}^{2}$ which is a $2-$almost-Riemannian structure ($2-$ARS for short) on $\mathbb{S}^{2}$ (see for instance \cite{Agra,Boscain,Prandi,Casarino,Colin,Strichartz} for more details). The pair $\left\{X,Y\right\}$ is called the generating frame of the $2-$ARS. 

In this work we are interested in the hypoelliptic operator defined by
\begin{equation}\label{intrinsec-grushin}
\mathcal L := \operatorname{div}_\mu \circ \nabla_{\text{sR}}=-X^{+}X-Y^{+}Y.
\end{equation}
Here, $\nabla_{\text{sR}}$ is the sub-Riemannian gradient defined by $\nabla_{\text{sR}} \phi = (X\phi)X+(Y\phi)Y$ for any $\phi\in C^\infty(\mathbb S^2)$, while $\operatorname{div}_\mu$ denotes the divergence w.r.t.\ the standard Riemannian volume form $\mu$ on $\mathbb S^2$, as induced by the Euclidean Lebesgue measure. Moreover, $X^{+}$ and $Y^{+}$ denote respectively the formal adjoints of $X$ and $Y$ taken in the space $ L^2(\mathbb{S}^2,\mu)$, the Hilbert space of measurable and
square-integrable functions over $\mathbb{S}^2$ with respect to $\mu$.
As it will be evident from the coordinate expression that we will present in the following, this is a degenerate operator that generalizes to the sphere $\mathbb S^2$ the Grushin operator. Note that, $\cC^\infty(\mathbb{S}^2)$ is canonically defined as the space of the restrictions to $\mathbb{S}^2$ of
functions that are $\cC^\infty$ on an open neighbourhood of $\mathbb{S}^2$. Such functions have compact support since $\mathbb{S}^2$ is a compact manifold.

The operator $\cL$ is essentially self-adjoint on $\cC^{\infty}(\mathbb{S}^{2})$, and we will henceforth consider its self-adjoint realization. See Section~\ref{s.Well}.
Our main result is then the following.
\begin{TheoPrinc}
	\label{thm:main-intrinsic}
Let $f_0\in  L^2(\mathbb S^2,\mu)$ and $u\in  L^2(0,T;  L^2(\mathbb S^2,\mu))$. Let $\overline{\omega} = \{ (x_1,x_2,x_3)\in \mathbb S^2\mid \alpha <|x_3|< \beta \}$ with $0\le\alpha<\beta<1$. We consider the following equation 
	\begin{equation}
	\label{eq:heat-intrinsic}
	\begin{cases}
	\partial_t f - \mathcal L f = u \mathbbm{1}_{\overline{\omega}},&\qquad \text{in }(0,T)\times \mathbb S^2,\\
	f|_{t=0} = f_0,& \qquad\text{in }\mathbb S^2.
	\end{cases}
	\end{equation}	
		 Then, the minimal time of null controllability from $\overline{\omega}$ satisfies $T_{\min}\ge  \log (1 / \sqrt{1-\alpha^2})$. Moreover, there exists $T^{*}>0$ such that, for every $T\ge T^{*}$,  system \eqref{eq:heat-intrinsic} is null controllable from $\overline{\omega}$ in time~$T$.
\end{TheoPrinc}
It should be noted that we will prove in this paper the Theorem \ref{thm:main-intrinsic} only in the interesting case when $\alpha>0$, i.e., when the control region $\overline{\omega}$ does not touch the degeneracy $\cE = \{x_3=0\}$. However, let us emphasise that if $\alpha=0$ then the equation \eqref{eq:heat-intrinsic} is null controllable in any time $T>0$. This result can be proved using a classical cut-off argument (as done for instance in \cite{Beau,Beau4} for the 2-dimensional parabolic Grushin operator) by taking advantage of the fact that the equation is null controllable in both hemispheres $\mathbb{S}_+^2$ and $\mathbb{S}_-^2$  in any time $T>0$ by the result of Lebeau and Robbiano \cite{Lebeau} (see also Fursikov and Imanuvilov \cite{Fursikov}) since  $\cL$ is a uniformly elliptic operator on the two hemispheres $\mathbb{S}_+^2$ and $\mathbb{S}_-^2$. 

To prove Theorem \ref{thm:main-intrinsic} and, by the way, understand how the operator $\mathcal L$ is connected with the Grushin operator, we use spherical coordinates. To this end, it is slightly easier to consider the vector fields \eqref{eq::vectorfields} as the restriction of $\mathbb{S}^2$ of the vector fields in $\RR^3$ given by the same formulae. 

Let
\begin{equation}\label{eq::Omega_domain}
\Omega :=  (-\pi/2,\pi/2)\times[0,2\pi),\qquad\mbox{and}\qquad U := \RR_+^{*}\times\Omega,
\end{equation}
and consider the latitude $x$ and longitude $y$ coordinates, viz.
\begin{eqnarray}\label{eq::diffeo F}
\operatorname{F}:&U&\longrightarrow\RR^3\nonumber\\
&(r,x,y)&\longmapsto\operatorname{F}(r,x,y) = (r\cos x\cos y,r\cos x\sin y,r\sin x),
\end{eqnarray}
so that $\operatorname{F}^{-1}(\mathbb{S}^{2}\bs\{N,S\})=V:=\{1\}\times\Omega\cong\Omega$. We let $\Phi := \operatorname{F}|_{V}$, then up to a rotation of angle $y-\frac{\pi}{2}$, the pull-back by $\Phi$ of the vector fields of the generating frame read
\begin{equation}\label{eq5}
\Phi^{*}X:= (d\operatorname{F}^{-1}\cdot X)|_{\Phi(r, x, y) }= \partial_x,\qquad \Phi^{*}Y := (d\operatorname{F}^{-1}\cdot Y)|_{\Phi(r, x, y) } = \tan x\partial_{y},
\end{equation}
where $d\operatorname{F}^{-1}$ denotes the inverse of the Jacobian matrix of $\operatorname{F}$.

Let us denote by $\cD(p)$ the linear span of the two vector fields $\Phi^{*}X$ and $\Phi^{*}Y$ at a point $\Phi^{-1}(p), \;p\in\mathbb{S}^2\bs\{N,S\}$. One can easily check that, $\cD(p)$ is $2-$dimensional except on the equator $\Phi^{-1}(\cE)  = \{x = 0\}$ where it is $1-$dimensional. We may observe also that, due to the system of coordinates, the vector field $\Phi^{*}Y$ is singular at $\pm\pi/2$.
The standard rotation-invariant measure on $\Omega$ in these coordinates is given by
\begin{equation}\label{eq7}
d\sigma = \cos xdxdy.
\end{equation}
\begin{figure}
	\centering
	\def\r{2.5}
	\tdplotsetmaincoords{60}{125}
	\begin{tikzpicture}[tdplot_main_coords]
	\draw[tdplot_screen_coords,thin,black!30] (0,0,0) circle (\r);
	\foreach \a in {-75,-44.5,...,75}
	{\tdplotCsDrawLatCircle[thin,black!25]{\r}{\a}}
	\foreach \a in {25,25.5,...,55}
	{\tdplotCsDrawLatCircle[thin,green!50]{\r}{\a}}
	\foreach \a in {-55,-54.5,...,-25}
	{\tdplotCsDrawLatCircle[thin,green!50]{\r}{\a}}
	\tdplotCsDrawLatCircle[red,thick]{\r}{0}
	\begin{scope}[thin,black!90]
	\draw[->] (0,0,0) -- (1.8*\r,0,0) node[anchor=north east] {$x_{1}$};
	\draw[->] (0,0,0) -- (0,1.5*\r,0) node[anchor=north] {$x_{2}$};
	\draw[->] (0,0,0) -- (0,0,1.5*\r) node[anchor=south east] {$x_{3}$};
	\draw (0,0,0) -- (0,0,-\r);
	\end{scope}
	\draw (0,0,0) node[above left]{$O$};
	\draw (0,2.7,2.5) node[above]{$\overline{\omega}$};
	\draw (1.5,-1,0) node[below]{$\cE$};
	\draw (0,0,\r) node[blue] {$\circ$};
	\draw (0,0,\r) node[above right]{$N$};
	\draw (0,0,-\r) node[blue] {$\circ$};
	\draw[dashed] (0,0,-\r) node[below]{$S$};
	\end{tikzpicture} 
	\caption{The equator $\cE$ (in red), a control region $\overline{\omega}$ (in green), north and south pole (in blue).} \label{fig1} 
\end{figure}
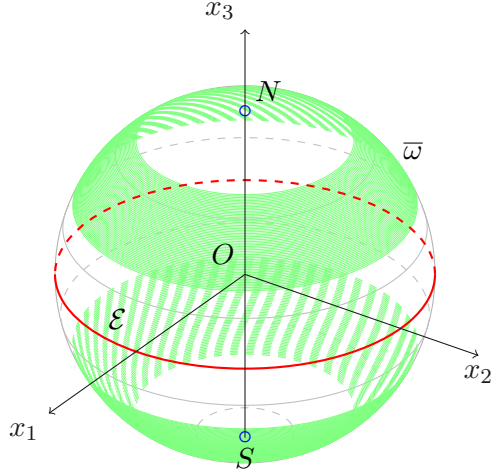
Observe that the diffeomorphism $\Phi :\Omega\to\mathbb{S}^2\bs\{N,S\}$ induces a unitary transformation
	\begin{eqnarray}\label{eq::unitary transformation}
	\operatorname{T_\Phi}:& L^2(\Omega,\sigma)&\longrightarrow  L^2(\mathbb{S}^2\bs\{N,S\},\mu)\\\nonumber
	&v&\longmapsto \operatorname{T_\Phi}v = v\circ\Phi^{-1},
	\end{eqnarray} 
	and that $\Phi^{*}X = \operatorname{T}_{\Phi}^{-1} X \operatorname{T}_{\Phi}$ and $\Phi^{*}Y =  \operatorname{T}_{\Phi}^{-1} Y \operatorname{T}_{\Phi}$. Here $\operatorname{T}_{\Phi}^{-1}$ is the inverse (or the adjoint) of $\operatorname{T_\Phi}$.
	From now on we let the \textit{spherical Grushin} operator be the coordinate representation under $\Phi$ of $\cL$. That is, the operator defined by
\begin{equation}\label{eq::intrinsic coordinate}
\Phi^{*}\mathcal{L}:= \operatorname{T}_{\Phi}^{-1}\cL\operatorname{T_\Phi}.
\end{equation}
In particular, $\Phi^{*}\cL$ is self-adjoint with core $\operatorname{T}_{\Phi}^{-1}(C^{\infty}(\mathbb(S^{2})))$. See Section~\ref{ss.Well1} for a characterization of its domain and corresponding boundary conditions.

In terms of the local generating family of vector fields \{$\Phi^{*}X$, $\Phi^{*}Y$\} we have
\begin{equation}\label{eq9}
\Phi^{*}\mathcal{L} := -(\Phi^{*}X)^{+}(\Phi^{*}X)-(\Phi^{*}Y)^{+}(\Phi^{*}Y) = \frac{1}{\cos x}\partial_{x}(\cos x\partial_{x})+\tan^2x\partial_{y}^{2},
\end{equation}
with $(\Phi^{*}X)^{+}$ and $(\Phi^{*}Y)^{+}$ being the formal adjoints of $\Phi^{*}X$ and $\Phi^{*}Y$ respectively, taken in the space $ L^2(\Omega,\sigma)$.

\begin{Rema}\label{Rema2}
	The singularity of (\ref{eq9}) at north and south poles is due to the latitude-longitude chart $\Phi$. We stress that whatever chart is chosen, this phenomenon of singularity will always occur since global coordinates do not exist\footnote{In fact, $\mathbb{S}^{2}$ is not a \textit{local} surface of $\RR^{3}$ in the sense of Berger and Gostiaux \cite[p. 348]{Berger_Gostiaux}, meaning that there is no open set $U\subset\RR^{2}$ and an immersion $\Phi\in\cC^ {\infty}(U;\RR^{3})$ such that $\Phi$ is a homeomorphism between $U$ and its image $\mathbb{S}^{2} = \Phi(U)$.} on $\mathbb{S}^{2}$.
\end{Rema}
Throughout the following, we let the real numbers $0<a<b\le\frac\pi 2$ be such that $\alpha=\sin a$ and $\beta=\sin b$, $\alpha$ and $\beta$ being as in Theorem~\ref{thm:main-intrinsic}. We set
\begin{equation}\label{eq::omega}
\omega := \omega_{a,b}\times(0,2\pi)\qquad\mbox{and}\qquad\omega_{a,b}:=(-b,-a)\cup(a,b).
\end{equation}
Hence, Theorem~\ref{thm:main-intrinsic} is equivalent to the following.
\begin{TheoPrinc}\label{Theo1}
	 Let $f_0\in  L^2(\Omega,\sigma)$ and $u\in  L^2(0,T;  L^2(\Omega,\sigma))$. Let $\omega$ be defined as in \eqref{eq::omega}. We consider the following equation
	\begin{equation}\label{eq10}
	\begin{cases}
	\displaystyle\partial_tf-(\Phi^{*}\cL)f = u\mathbbm{1}_\omega,&\mbox{ in }(0,T)\times\Omega,\cr
	\displaystyle f|_{t=0} = f_{0},&\mbox{ in }\Omega.
	\end{cases}
	\end{equation}
	Then, the minimal time of null controllability from $\omega$ satisfies $T_{\min}\ge \log(1/\cos a)$. Moreover, there exists $T^{*}>0$ such that, for every $T\ge T^{*}$, system \eqref{eq10} is null controllable from $\omega$ in~time~$T$.
\end{TheoPrinc}

 Recall that system \eqref{eq10} is \textit{null controllable} from $\omega\subset\Omega$ in time $T>0$ if, for every $f_0\in  L^2(\Omega,\sigma)$, there exists a control $u\in  L^2(0,T; L^2(\Omega,\sigma))$ supported in $(0,T)\times\omega$ such that the solution $f$ of \eqref{eq10} satisfies $f(T,\cdot,\cdot) = 0$.

The following remark, although formal, illustrates the connection between the spherical Grushin operator $\Phi^{*}\cL$ and the 2D Grushin operator $\cG$.
\begin{Rema}\label{Rema3}
We consider \eqref{eq7} and \eqref{eq9}. 
Taking the first order Taylor expansion of $\cos x$ and $\tan x$ at  $x \approx 0$ we observe that $d\sigma\approx dxdy$ and $\Phi^{*}\mathcal{L}\approx \dx^{2}+x^2\dy^{2}$, so that $\Phi^{*}\mathcal{L}$ behaves like the Grushin operator (\ref{eq1}) in a neighbourhood of the degeneracy. As a consequence, we may expect the same properties of null controllability for the parabolic equation \eqref{eq10} associated to $\Phi^{*}\mathcal{L}$ as for the $2D$ parabolic Grushin equation (\ref{eq2}). 
\end{Rema}

Let us briefly discuss our strategy of proof.
As it is now classical, Theorem \ref{Theo1} is a straightforward consequence of the Hilbert Uniqueness Method \cite{Lions} meaning that the null controllability property of system \eqref{eq10} is equivalent to the observability of the adjoint system associated to \eqref{eq10}.

Thus, Theorem \ref{Theo1} is equivalent to the following.
\begin{TheoPrinc}\label{Theo2}
	 Let $g_0\in  L^2(\Omega,\sigma)$ and $\omega$ be defined as in \eqref{eq::omega}. Consider the adjoint system of~\eqref{eq10},
	\begin{equation}\label{eq12}
	\begin{cases}
	\partial_tg-(\Phi^{*}\cL)g = 0,&\mbox{ in }(0,T)\times\Omega,\cr
	g|_{t=0} = g_{0},&\mbox{ in }\Omega.
	\end{cases}
	\end{equation}
Then, the minimal time required for observability
in $\omega$ satisfies $T_{\min}\ge\log(1/\cos a)$. Moreover, there exists $T^{*}>0$ such that, for every $T\ge T^{*}$, system \eqref{eq12} is observable in $\omega$ in time $T$.
\end{TheoPrinc}

Recall that system \eqref{eq12} is \textit{observable} in $\omega\subset\Omega$ in time $T>0$ if there exists $C(T,\omega)>0$ such that, for every $g_{0}\in  L^2(\Omega,\sigma)$, the solution $g$ of system \eqref{eq12}
satisfies
\begin{equation}\label{eq13}
\int_{\Omega}|g(T,x,y)|^2d\sigma\le C(T,\omega)\int_{0}^{T}\int_{\omega}|g(t,x,y)|^2d\sigma dt.
\end{equation}
The observability inequality (\ref{eq13}) means that the energy of the solution of (\ref{eq12}) concentrated in $\omega$ yields an upper bound of the energy in time $T$ everywhere in $\Omega$.

Finally, note that the proof of Theorem \ref{Theo2} is divided in two distinct steps:
\begin{enumerate}
	\item Prove that for any time $T\le\log(1/\cos a)$, the equation is not observable in $\omega$ in time $T$;
	\item Prove that there exists $T^{*}>0$ such that, for every $T\ge T^{*}$, the equation is observable in $\omega$ in time $T$.
\end{enumerate}

\subsection{Comments and open questions}\label{ss.intro4}
The lower bound of minimal time $T_{\mbox{\scriptsize{min}}}$ in Theorem \ref{Theo1} appears to be the Agmon distance between the control region $\omega = \omega_{a,b}\times(0,2\pi)$ and the degeneracy (the equator in fact) $\Phi^{-1}(\cE) = \{x=0\}$, that is to say, the lower bound of $T_{\mbox{\scriptsize{min}}}$ satisfies (\ref{eq4}) with $q(x) = \tan x$ and $L_{+}=a$. However, it should be noted that the result of \cite{Beau2} can not be directly applied to our case. 

In fact,  in \cite{Beau2}, the authors investigate the \textit{boundary null-controllability} of the $2D$ parabolic Grushin equation with more general potential satisfying assumptions \eqref{eq3}. In our case, equation \eqref{eq10} is the coordinate representation of equation \eqref{eq:heat-intrinsic}, posed on the whole sphere $\mathbb{S}^{2}$, which is a compact manifold without boundary. Moreover the potential $q(x) = \tan x$ does not satisfies assumptions \eqref{eq3} on $(-\pi/2,\pi/2)$.

The presence of two spherical crowns in the control region is a technical assumption required in the proof of the Carleman estimate in Section \ref{s.Carleman} and is related to the symmetry of the singularity locations (the north and south poles) with respect to the equator $\Phi^{-1}(\cE)  = \{x=0\}$. The case where the control function acts only on one spherical crown (i.e., $\omega=(a,b)\times(0,2\pi)$) remains open. Another interesting open question is to show the sharpness of the lower bound of the minimal time, as done in \cite{Beau4} for the $2D$ parabolic Grushin equation.  
\begin{Rema}\label{Rema6}
	If the elevation angle (latitude) $a$ of the control region $\omega$ with respect to equator is equal to zero, i.e., if $\omega$ contains the equator, then the strategy used in Subsection \ref{ss.proof_negative} to obtain the lower bound of the minimal time $T_{\mbox{\scriptsize{min}}}$ can not be applied. On the other hand, if the control acts only on one spherical crown (i.e., $\omega = (a,b)\times (0,2\pi)$), the proof presented here still applies and shows that the lower bound of minimal time is still $\log(1/\cos a)$.
\end{Rema}

\subsection{Structure of the paper}\label{ss.intro5}
The first part of the paper, contained in Section \ref{s.Well}, is devoted to general results about parabolic Grushin equation (\ref{eq10}). Here we prove the well-posedness of equation in Subsection \ref{ss.Well1}, we study the properties of Fourier components of solution of the adjoint system (\ref{eq12}) in Subsection \ref{ss.Well2} as well as their dissipation rate. In Subsection \ref{ss.Observability}, we present the strategy of the proof of Theorem \ref{Theo2}, that is, we show how the uniform observability estimate of Fourier components yields the observability estimate of solution of the adjoint system.

In Section \ref{s::weight}, we recast the equation satisfied by the Fourier components in spaces $ L^2$ without weight using unitary transformations.

In Section \ref{s.Carleman} we prove a global Carleman estimate for the $1D$ parabolic equation satisfy by the Fourier components for non zero frequencies in a space $ L^2$ without weight.

Finally, Section \ref{s.proof} is devoted to the proof of Theorem \ref{Theo2} (or equivalently of Theorem \ref{Theo1} and therefore the proof of Theorem \ref{thm:main-intrinsic}). In Subsection \ref{ss.proof_negative} we prove the negative statement of Theorem \ref{Theo2} and the positive statement in Subsection \ref{ss.proof_positive}.

\paragraph{Acknowledgement.}
This work started when the author was a Master 2 student at the Institut de Mathématiques et de Sciences Physiques in Dangbo, Benin. The author would like to thank its thesis advisors Yacine Chitour and Dario Prandi for bringing this problem to its attention and for interesting and fruitful discussions. He would like to thank also the reviewers for the careful reading of the paper and for their valuable comments having enhanced the presentation of the paper.

\section{Well-posedness, Fourier decomposition and strategy for the proof}\label{s.Well}
\subsection{Well-posedness of Cauchy problems}\label{ss.Well1}

It is interesting and useful to start with the well-posedness of the parabolic equation \eqref{eq:heat-intrinsic} associated to the intrinsic operator $\cL$ as defined in \eqref{intrinsec-grushin}. Since $\left\{X,Y,[X,Y]\right\}_p$ generates the tangent space $T_p\mathbb{S}^{2}$ for any $p\in \mathbb S^2$, it follows from Strichartz \cite[p.260-261]{Strichartz} that $-\cL$ with domain
\begin{equation}\label{eq::intrisic-domain}
\operatorname{D}(\cL) = \big\{u\in L^2(\mathbb{S}^2,\mu):\cL u:=-X^{+}Xu-Y^{+}Yu\in L^2(\mathbb{S}^2,\mu)\big\},
\end{equation}
is a densely-defined, self-adjoint operator on $ L^2(\mathbb{S}^2,\mu)$, hypoelliptic \cite[Theorem 1.1]{Hormander} and has a compact resolvent. Therefore, its spectrum is real, discrete and consists of eigenvalues with finite multiplicity, labelled in increasing order, that is, $(\lambda_m)_{m\in\NN^{*}}$, with $0=\lambda_1<\lambda_2\le\cdots\le\cdots$, with $\lambda_m\rightarrow\infty$ as $m\rightarrow\infty$. Moreover, there exists an orthonormal Hilbert basis $(\varphi_m)_{m\in\NN^{*}}$ of $ L^2(\mathbb{S}^2,\mu)$ consisting of eigenfunctions of $\cL$ associated with the eigenvalues $(\lambda_m)_{m\in\NN^{*}}$.

\begin{Rema}
	It should be noted that the sub-Riemannian manifold $\mathbb{S}^2$ endowed with the $2-$ARS described in Subsection \ref{ss.intro2} is obtained as restriction of complete Riemannian manifolds. So, it is completes as metric space. It follows that, the sub-Riemannian Laplacian $\cL$ defined on $\cC^\infty(\mathbb{S}^2)$ is essentially self-adjoint in $ L^2(\mathbb{S}^2,\mu)$ and the domain of its unique self-adjoint extension coincides with \eqref{eq::intrisic-domain} (see, Strichartz \cite[p.261]{Strichartz}, \cite[p.50 and Theorem 2.4]{Strichartz_2}).
\end{Rema}

We define the intrinsic semigroup on $ L^2(\mathbb{S}^2,\mu)$ denoted $(e^{t\cL})_{t\ge 0}$, as the family of operator $ L^2(\mathbb{S}^2,\mu)\rightarrow L^2(\mathbb{S}^2,\mu)$ defined as follows for every $t\ge 0$: given $f_0\in L^2(\mathbb{S}^2,\mu)$, $e^{t\cL}f_0$ is the unique solution at time $t$ of the homogeneous equation of \eqref{eq:heat-intrinsic}, which is $\cC^\infty$ on $]0,+\infty[\times\mathbb{S}^2$ (by the hypoellipticity of operator $\cL$) and given by
\begin{equation}\label{eq::intrisic-semigroup}
e^{t\cL}f_0 = \sum\limits_{m\in\NN^{*}}e^{-t\lambda_{m}}\langle f_0,\varphi_m\rangle_{ L^2(\mathbb{S}^2,\mu)}\varphi_m.
\end{equation}

Let us state the following well-posedness result of the intrinsic parabolic equation \eqref{eq:heat-intrinsic} whose proof is classical (see, e.g., \cite[Chapter 4]{Pazy}).

\begin{Prop}
	Given $T>0$, $f_0\in L^2(\mathbb{S}^2,\mu)$ and $v := \mathbbm{1}_\omega u\in L^2(0,T; L^2(\mathbb{S}^2,\mu))$, there exists a unique solution
	$$
	f\in\cC([0,T]; L^2(\mathbb{S}^2,\mu))\cap L^2((0,T);\operatorname{D(\cL)})
	$$
	of equation \eqref{eq:heat-intrinsic}, and $f$ is given by Duhamel's formula
	\begin{equation}
	f(t) = e^{t\cL}f_{0}+\int_{0}^{t}e^{(t-s)\cL}v(s)ds,\hspace{0.5cm}t\in[0,T].
	\end{equation}
\end{Prop}

We now can provide an argument about the well-posedness of the parabolic equation \eqref{eq10} associated with the spherical Grushin operator $\Phi^{*}\cL$ defined in \eqref{eq::intrinsic coordinate} (or equivalently in \eqref{eq9}). 

Let $\operatorname{H_{\sigma}} :=  L^2(\Omega,\sigma)$, and denote by $\langle\cdot,\cdot\rangle_{\operatorname{H_{\sigma}}}$ and $\|\cdot\|_{\operatorname{H_{\sigma}}}$, respectively, the scalar product and norm in $\operatorname{H_{\sigma}}$.We have that $(\cL,\operatorname{D}(\cL))$ and $(\Phi^{*}\cL,\operatorname{D}(\Phi^{*}\cL))$ are unitarily equivalent, where we let
\begin{equation}
\Phi^{*}\mathcal{L}= \operatorname{T}_{\Phi}^{-1}\cL\operatorname{T_\Phi}\qquad\mbox{on}\qquad\operatorname{D}(\Phi^{*}\cL)=\operatorname{T}_{\Phi}^{-1}(\operatorname{D}(\cL)).
\end{equation}
Here, $\operatorname{T_\Phi}$ is the unitary transformation defined in \eqref{eq::unitary transformation}, $\operatorname{T}_{\Phi}^{-1}$ being is inverse. So, $\Phi^{*}\cL$ with domain $\operatorname{D}(\Phi^{*}\cL)$ is a densely-defined, self-adjoint operator on $\operatorname{H_{\sigma}}$ and has compact resolvent.  We also remark that $v\in\operatorname{D}(\Phi^{*}\cL)$ means $v=u\circ\Phi$ for some $u\in \operatorname{D}(\cL)$. So, we have the following.
\begin{Lem}\label{Lem::bdry conditions in 2D}
  Let $v\in \operatorname{D}(\Phi^{*}\cL)$. Then, we have that $v,\;\Phi^{*}\cL v\in\operatorname{H_{\sigma}}$, the function $y \mapsto v(\pi/2,y)$ (resp.~$y\mapsto v(-\pi/2,y)$) is constant, and $y\mapsto v(x,y)$ is $2\pi$-periodic for any $x\in [-\pi/2,\pi/2]$. Moreover, the following functions are well defined and real-valued: 
	\begin{equation}\label{eq::bdry conditions in 2D}
	  y\in[0,2\pi)\mapsto \dx v(\pi/2, y), \quad y\in[0,2\pi)\mapsto \dx v(-\pi/2, y), \quad (x,y)\in\Omega \mapsto\tan x\,\dy v(x,y).
	\end{equation}
\end{Lem}
\begin{Rema}
	We stress that boundary conditions of Lemma~\ref{Lem::bdry conditions in 2D} are naturally associated to Cauchy problems \eqref{eq10} and \eqref{eq12}.
\end{Rema}

Let $\{W_{\ell,n}\}_{\ell\in\NN,-\ell\le n\le\ell}$ denotes the family of spherical harmonics, defined by
\begin{equation}\label{eq22}
W_{\ell,n}(x,y) = \sqrt{\frac{2\ell+1}{4\pi}\frac{(\ell-n)!}{(\ell+n)!}}P_\ell^n(\sin x)e^{iny},\hspace{1cm}\forall x\in [-\pi/2,\pi/2]\times[0,2\pi),
\end{equation}
with $P_{\ell}^{n}$ being associated Legendre functions of the first kind. Then we can check that each $W_{\ell,n}$ lies in $\operatorname{D(\Phi^{*}\cL)}$ and the operator $\Phi^{*}\cL$ satisfies (see, \cite[p.9]{Casarino} and references within)
\begin{equation}\label{eq17}
-(\Phi^{*}\cL) W_{\ell,n} = \lambda_{\ell,n}W_{\ell,n},
\qquad\qquad\lambda_{\ell,n} := \ell(\ell+1)-n^2,\qquad\forall\;|n|\le \ell\in\NN.
\end{equation}
Moreover, by using the identification $\operatorname{H_{\sigma}}\cong  L^2((-\pi/2,\pi/2);\cos xdx)\otimes L^2((0,2\pi),dy)$, we have that $\{W_{\ell,n}\}_{\ell\in\NN,-\ell\le n\le\ell}$ form an orthonormal Hilbert basis of the space $\operatorname{H_{\sigma}}$ \cite[p.137]{Stein} so that, $\operatorname{D(\Phi^{*}\cL)}$ is a non-empty and dense subspace of $\operatorname{H_{\sigma}}$. 

The spherical Grushin semigroup on $\operatorname{H_{\sigma}}$ denoted $(e^{t\Phi^{*}\cL})_{t\ge 0}$ is then the family of operators $\operatorname{H_{\sigma}}\rightarrow\operatorname{H_{\sigma}}$ defined as follows for every $t\ge 0$: given $f_0\in\operatorname{H_{\sigma}}$, $e^{t\Phi^{*}\cL}f_0$ is the unique solution at time $t$ of the homogeneous equation of \eqref{eq10}, which is $\cC^\infty$ on $]0,+\infty[\times\Omega$ and given by
\begin{equation}\label{eq::spherical Grushin semigroup}
e^{t\Phi^{*}\cL}f_0 = \sum\limits_{|n|\le \ell\in\NN}e^{-t\lambda_{\ell,n}}\langle f_0,W_{\ell,n}\rangle_{\operatorname{H_{\sigma}}}W_{\ell,n}.
\end{equation}

We now can state the following well-posedness result of the parabolic equation \eqref{eq10} associated with the spherical Grushin operator $\Phi^{*}\cL$ (see, e.g., \cite[Chapter 4]{Pazy}).
\begin{Prop}
	Given $T>0$, $f_0\in\operatorname{H_{\sigma}}$ and $v := \mathbbm{1}_\omega u\in L^2(0,T;\operatorname{H_{\sigma}})$, there exists a unique solution
	$$
	f\in\cC([0,T];\operatorname{H_{\sigma}})\cap L^2((0,T);\operatorname{D(\Phi^{*}\cL)})
	$$
	of equation \eqref{eq10}, and $f$ is given by Duhamel's formula
	\begin{equation}\label{eq21}
	f(t) = e^{t\Phi^{*}\cL}f_{0}+\int_{0}^{t}e^{(t-s)\Phi^{*}\cL}v(s)ds,\hspace{0.5cm}t\in[0,T].
	\end{equation}
\end{Prop}
We end this section by the following Hardy-Poincaré inequality in the Sobolev space $ H_0^1(-\pi/2,\pi/2)$. The reader could find another proof of such inequality in \cite[p.92]{Opic}.
\begin{Lem}\label{lem:Hardy-Poincaré}
	Let $w\in  H_0^1(-\pi/2,\pi/2)$. Then, it holds
	\begin{equation}\label{eq::Poncaré-Hardy}
		\int_{-\frac\pi 2}^{\frac\pi 2}\frac{|w(x)|^2}{\cos ^2x}dx\le 4\int_{-\frac\pi 2}^{\frac\pi 2}|w'(x)|^2dx.
	\end{equation}
	\begin{proof}
			First of all, highlight that, the main theorem of \cite[p. 199]{Chisholm} is valid in the space $\displaystyle L^2([0,\pi/2))$ by repeating the same proof with $\pi/2$ playing the role of $\infty$. Note that, such results are consequence of lemma of \cite[p.42]{Everitt}, adapting the proof in the case at hand. Let $f\in\displaystyle L^2(-\pi/2,\pi/2)$. Then,
		$$
		\int_{0}^{\frac\pi 2}(|f(x)|^2+|f(-x)|^2)dx = \int_{-\frac\pi 2}^{\frac\pi 2}|f(x)|^2dx,
		$$
		so that, $f, \widetilde{f}\in\displaystyle L^2([0,\pi/2))$, where $\widetilde{f}(x) = f(-x), x\in[0,\pi/2)$. We adopt the notations of \cite{Chisholm}, then \cite[eq.(1.3)]{Chisholm} recasts
		\begin{equation}\label{eq::operator S}
		(\operatorname{S}f)(x) = \phi(x)\int_{x}^{\frac\pi 2}\psi(t)f(t)dt,\qquad\qquad x\in[0,\pi/2).
		\end{equation}
		We let $\displaystyle\phi(x) = 1/\cos x$ and $\psi(x) = 1$ for all $x\in[0,\pi/2)$. Thus, $\psi\in L^2([0,\pi/2))$, $\phi\in L^2([0,\alpha])$ for all $0<\alpha<\pi/2$ and 
		$$
		\int_{0}^{x}|\phi(t)|^2dt\int_{x}^{\frac\pi 2}|\psi(t)|^2dt = \left(\frac{\pi}{2}-x\right)\tan x\le 1,\qquad\forall x\in[0,\pi/2],
		$$
		so that, \cite[eq. (2.1) to (2.3)]{Chisholm} are satisfy with $K:=1$. It follows that $\operatorname{S}$ defined in \eqref{eq::operator S} is a bounded operator from $ L^2([0,\pi/2))$ to itself and the following holds for all $f\in\displaystyle L^2([0,\pi/2))$,
		\begin{equation}\label{eq:: norm of S}
		\int_{0}^{\frac\pi 2}|(\operatorname{S}f)(x)|^2dx\le 4K\int_{0}^{\frac\pi 2}|f(x)|^2dx = 4\int_{0}^{\frac\pi 2}|f(x)|^2dx.
		\end{equation}
		Let now $w\in H_0^1(-\pi/2,\pi/2)$, then $w, w' \in \displaystyle L^2(-\pi/2,\pi/2)$ and $w(\pm\pi/2) = 0$. Letting $f=w'$ in \eqref{eq::operator S} we find
		$$
		(\operatorname{S}w')(x) = \phi(x)\int_{x}^{\frac\pi 2}\psi(t)w'(t)dt = -\frac{w(x)}{\cos x},\qquad x\in[0,\pi/2).
		$$
		Hence \eqref{eq:: norm of S} leads to
		$$
		\int_{0}^{\frac\pi 2}\frac{|w(x)|^2}{\cos ^2x}dx\le 4\int_{0}^{\frac\pi 2}|w'(x)|^2dx.
		$$
		We argue similarly for $\widetilde{w}$, and combining both inequalities we complete the proof of lemma.
	\end{proof}
\end{Lem}

\subsection{Fourier decomposition of solution}\label{ss.Well2}
Using a complete orthonormal eigenbasis of $ L^2((0,2\pi),dy)$, we can separate the space $\operatorname{H_{\sigma}} = \oplus_{n\in\ZZ}^{\perp}\cH_{n}$, where $\cH_{n}\cong  L^2((-\pi/2,\pi/2);\cos xdx)$. Therefore, one has for every $t\ge 0$,
\begin{equation}\label{eq::semigroup diagonalisation}
e^{t\Phi^{*}\cL} = \bigoplus_{n\in\ZZ}^{\perp}e^{t\cL_n},
\end{equation}
where for any $n\in\ZZ$, the operator $\cL_n$ is defined on $\cH_n$ by
\begin{eqnarray}\label{eq27}
\operatorname{D}(\cL_{n}) = \left\{v\in\cH_n:\cL_{n}v\in\cH_n,\;v(\pm\pi/2),\;v'(\pm\pi 2)\in\RR\right\},
\end{eqnarray}
\begin{equation}\label{eq26}
\cL_{n}v = \frac{1}{\cos x}(\cos xv')'-n^{2}\tan^{2}xv,\qquad\forall v\in \operatorname{D}(\cL_{n}).
\end{equation}

Since the solution $g$ of \eqref{eq12} belongs to $\cC([0,T];\operatorname{H_{\sigma}})$, the function $y\mapsto g(t,x,y)$ belongs to $ L^2((0,2\pi),dy)$ for a.e. $(t,x)\in(0,T)\times(-\pi/2 ,\pi/2)$. So, the adjoint system \eqref{eq12} is formally equivalent to the following family of one-dimensional parabolic equations indexed by $n\in\ZZ$,
\begin{equation}\label{eq23}
\begin{cases}
\partial_tg_n-\cL_ng_n = 0,&\mbox{ a.e. in } (0,T)\times(-\pi/2,\pi/2),\cr
g_n(0,x) = g_{0,n}(x),&\;x\in (-\pi/2,\pi/2).
\end{cases}
\end{equation}
Here, the n-th Fourier component $g_n$ is given by
\begin{equation}\label{eq24}
g_n(t,x) = \int_{0}^{2\pi}g(t,x,y)e^{iny}dy,\qquad(t,x)\in(0,T)\times(-\pi/2,\pi/2).
\end{equation}
We derive in the following lemmas some properties of functions belonging to $\operatorname{D}(\cL_{n})$ as well as their behaviour at $\pm\pi/2$. We begin by the case $n=0$.
\begin{Lem}\label{lem::fourier 1}
	Let $v\in\operatorname{D}(\cL_0)$. Then $v$  belongs to the Sobolev space $ H^1(-\pi/2,\pi/2)$ and $v$ is locally absolutely continuous on $[-\pi/2,\pi/2]$. Moreover, it holds
	\begin{equation}\label{eq33}
	\lim\limits_{x\rightarrow-\frac{\pi}{2}^{+}}v'(x)=\lim\limits_{x\rightarrow\frac{\pi}{2}^{-}}v'(x)=0.
	\end{equation}
\end{Lem}
\begin{proof}
	Let $v\in \operatorname{D}(\cL_0)$. Then  $v,\;\cL_0v\in\cH_0\cong L^2((-\pi/2,\pi/2);\cos xdx)$ and $v(\pm\pi/2), v'(\pm\pi/2)\in~\RR$. One has, 
	\begin{equation}\label{eq::A0}
	\|v'\|_{\cH_0}^2 = -\langle\cL_0v,v\rangle_{\cH_0}<\infty.
	\end{equation}
	Since it is clear that $\sin xv', \cos xv, \sqrt{\cos x}v, \sqrt{\cos x}v'\in\cH_0$, it holds
	\begin{equation}\label{eq::A02}
	\|v\|_{ H^1(-\pi/2,\pi/2)}^2 = 2\left[\langle\cL_0v,\cos xv-\sin xv'\rangle_{\cH_0}+\|\sqrt{\cos x}v\|_{\cH_0}^2+\|\sqrt{\cos x}v'\|_{\cH_0}^2	\right]<\infty.
	\end{equation}
	It follows that $v$ belongs to the Sobolev space $ H^1(-\pi/2,\pi/2)\hookrightarrow W^{1,1}(-\pi/2,\pi/2)$, and then $v$ is locally absolutely continuous on $[-\pi/2,\pi/2]$. On the other hand, one has
	\begin{eqnarray}\label{eq::second derivative}
	\int_{-\frac\pi 2}^{\frac\pi 2}\frac{|v'(x)|^2}{\cos x}dx&\le&
	\int_{-\frac\pi 2}^{\frac\pi 2}\left(|v''(x)|^2+\frac{|v'(x)|^2}{\cos^2x}\right)\cos xdx\nonumber\\
	&=&\int_{-\frac\pi 2}^{\frac\pi 2}|\cL_0v(x)|^2\cos xdx+\left|v'(\pi/2)\right|^2+\left|v'(-\pi/2)\right|^2<\infty.
	\end{eqnarray}
	Since $\cos(\pm{\pi}/{2})=0$ this implies \eqref{eq33}. Moreover \eqref{eq::second derivative} shows also that $v''\in\cH_0$. In particular,  $\tan xv'\in\cH_0$, since $\cL_0v\in\cH_0$ and $v''\in\cH_0$. This completes the proof of lemma.
\end{proof}


In the case $n\in\ZZ\bs\{0\}$, we have the following
\begin{Lem}\label{Fourier.lem1}
	Let $n\in\ZZ\bs\{0\}$ and $v\in \operatorname{D}(\cL_{n})$. Then $v$ belongs to the Sobolev space $ H^1(-\pi/2,\pi/2)$ and $v$ is locally absolutely continuous on $\displaystyle[-\pi/2,\pi/2]$. Moreover, it holds
	\begin{equation}\label{eq34}
	\lim\limits_{x\rightarrow-\frac{\pi}{2}^{+}}v(x)=\lim\limits_{x\rightarrow\frac{\pi}{2}^{-}}v(x)=0\qquad\mbox{and}\qquad\lim\limits_{x\rightarrow-\frac{\pi}{2}^{+}}\frac{v(x)}{\sqrt{\cos x}}  = \lim\limits_{x\rightarrow\frac{\pi}{2}^{-}}\frac{v(x)}{\sqrt{\cos x}}=0.
	\end{equation}
\end{Lem}
\begin{proof}
	Let $n\in\ZZ\bs\{0\}$ and $v\in \operatorname{D}(\cL_{n})$. Since  $v,\;\cL_{n}v\in\cH_n\cong L^2((-\pi/2,\pi/2);\cos xdx)$ and $v(\pm\pi/2),\;v'(\pm\pi 2)\in\RR$, one has,
	\begin{equation}\label{eq::A1}
	\|v'\|_{\cH_n}^2 \le \int_{-\frac\pi 2}^{\frac\pi 2}\left(|v'(x)|^2+|n\tan xv|^2\right)\cos xdx = -\langle v,\cL_{n}v\rangle_{\cH_n}<\infty,
	\end{equation}
	showing in particular that $\tan xv\in\cH_n$. It follows that
	\begin{equation}\label{eq::lem useful 1}
	\|\cos^{-1}xv\|_{\cH_n}^2 = \|\tan xv\|_{\cH_n}^2+\|v\|_{\cH_n}^2<\infty,
	\end{equation}
	and by Cauchy-Schwarz' inequality,
	\begin{equation}\label{eq::A11}
	\|v\|_{ L^2(-\pi/2,\pi/2)}^2\le\|\cos^{-1}xv\|_{\cH_n}\|v\|_{\cH_n}<\infty.
	\end{equation}
	So,
		\begin{equation}\label{eq::lem useful}
	\|\cos^{-1}v\|_{ L^2(-\pi/2,\pi/2)}^2 = \|v\|_{ L^2(-\pi/2,\pi/2)}^2+\frac{1}{n^2}\left[\left\langle-\cL_{n}v,\frac{v}{\cos x}\right\rangle_{\cH_n}-\int_{-\frac\pi 2}^{\frac\pi 2}\tan x|v(x)|^2dx\right]<\infty.
	\end{equation}
	It is clear that $\sin xv'\in\cH_n$, so that,
	\begin{equation}\label{eq::A12}
	\|v'\|_{ L^2(-\pi/2,\pi/2)}^2 = -2\langle\cL_nv,\sin xv'\rangle_{\cH_n}+n^2\int_{-\frac\pi 2}^{\frac\pi 2}\sin^2x(3+\tan^2x)|v(x)|^2dx<\infty,
	\end{equation}
	by \eqref{eq::lem useful 1}, \eqref{eq::A11} and \eqref{eq::lem useful}. Thus,  $v$ belongs to
	 the Sobolev space $ H^1(-\pi/2,\pi/2)\hookrightarrow W^{1,1}(-\pi/2,\pi/2)$, and then $v$ is locally absolutely continuous on $\displaystyle[-\pi/2,\pi/2]$. So,
	\begin{equation}\label{eq::representation}
	v(x_2)-v(x_1) = \int_{x_1}^{x_2}v'(s)ds\qquad \forall x_1,x_2\in[-\pi/2,\pi/2].
	\end{equation}

	 On the other hand, since  $\tan xv\in\cH_n$, the first identity of \eqref{eq34} immediately follows. Let us turn to an argument for the second identity of \eqref{eq34}. Let $\varepsilon>0$, then by the first identity of \eqref{eq34}, and \eqref{eq::representation} one has for all $v\in \operatorname{D}(\cL_{n})$, $n\neq 0$,
	$$\left|v\left(-\frac\pi 2+\varepsilon\right)\right|\le\int_{-\frac\pi 2}^{-\frac\pi 2+\varepsilon}|v'(t)|dt\le\|v'\|_{\infty}\varepsilon.$$
	Hence,
	$$\lim\limits_{x\rightarrow-\frac{\pi}{2}^{+}}\frac{|v(x)|}{\sqrt{\cos x}}=\lim\limits_{\varepsilon\rightarrow0}\frac{|v(-\frac\pi 2+\varepsilon)|}{\sqrt{\cos(-\frac\pi 2+\varepsilon)}}\le\|v'\|_{\infty}\lim\limits_{\varepsilon\rightarrow0}\frac{\varepsilon}{\sqrt{\sin\varepsilon}} = 0.$$
	The proof of the limit at $\pi/2$ is similar.
\end{proof}

\begin{Rema}\label{rmq::useful}
	Lemmas \ref{lem::fourier 1} and \ref{Fourier.lem1} show in particular that, for all $v\in \operatorname{D}(\cL_n)$, $\cL_nv$ has a meaning $\operatorname{a.e.}$ in $(-\pi/2,\pi/2)$. Moreover, Lemma \ref{Fourier.lem1} also shows that the domain $\operatorname{D(\cL_{n})}$ is a subspace of the Sobolev space $\displaystyle H_0^1(-\pi/2,\pi/2)$ in the case $n\in\ZZ\bs\{0\}$. Therefore, Lemma~\ref{lem:Hardy-Poincaré} holds true in $\operatorname{D(\cL_{n})}$.
\end{Rema}
%

The following proposition is the direct consequence of the section \ref{ss.Well1}.  We refer also to \cite[p. 68]{Naimark} in which the theory of singular Sturm-Liouville equation is well-elaborated. 
\begin{Prop}\label{fourier.prop1}
	Let $n\in\ZZ$. Then, $-\cL_{n}:\operatorname{D}(\cL_{n})\subset \cH_{n}\rightarrow\cH_{n}$ is a densely defined, self-adjoint positive operator with compact resolvent.
\end{Prop}

One can check that the functions $v_{n,\ell}$ defined for $\ell\in\NN$, $n\in\ZZ$ and $|n|\le\ell$ by
\begin{equation}\label{eq28}
v_{n,\ell}(x) = \sqrt{\frac{2\ell+1}{2}\frac{(\ell-n)!}{(\ell+n)!}}P_\ell^n(\sin x),\hspace{1cm}\forall x\in [-\pi/2,\pi/2],
\end{equation}
form a complete orthonormal set of the Hilbert space $\cH_{n}$, with $P_{\ell}^{n}$ being the associated Legendre function of the first kind. Moreover, each $v_{n,\ell}$ lies in $\operatorname{D}(\cL_{n})$ and we have $$-\cL_{n}v_{n,\ell} = (\ell(\ell+1)-n^2)v_{n,\ell}.$$ 
So the functions $v_{n,\ell}$
are the eigenfunctions of operators $-\cL_{n}$ with eigenvalues $\lambda_{\ell,n} = \ell(\ell+1)-n^2$.

Thanks to Proposition \ref{fourier.prop1},
it is then straightforward to prove the following
\begin{Prop}\label{fourier.prop2}
	Let $T>0$. For every $n\in\ZZ$, the n-th Fourier component $g_n$ of $g$, as given by (\ref{eq24}), is the unique solution of \eqref{eq23} lying in the class
	\begin{equation}\label{eq30}
	\cC([0,T];\cH_{n})\cap\cC((0,T);\operatorname{D}(\cL_{n}))\cap\cC^{1}((0,T);\cH_{n}).
	\end{equation}
	Moreover, it is equal to
	\begin{equation}
	e^{t\cL_n}g_{0,n} = \sum\limits_{\ell\in\NN}e^{-\lambda_{\ell,n}t}\langle g_{0,n},v_{n,\ell}\rangle_{\cH_n}v_{n,\ell},
	\end{equation}
	where $g_{0,n}\in\cH_n$ is given by $g_{0,n}(x) = \displaystyle\int_{0}^{2\pi}g_{0}(x,y)e^{iny}dy$, and $g_0$ being the initial condition in equation \eqref{eq12}.
\end{Prop}
\begin{Rema}\label{fourier.rmq2}
	In fact, we may show by an inductive argument that for all $n\in\ZZ$
	\begin{equation}\label{eq31}
	g_n\in\cC^{\infty}((0,T);\operatorname{D}(\cL_{n})).
	\end{equation}
	Moreover, $g_n$ is $\cC^\infty$ on $\displaystyle]0,+\infty[\times(-\pi/2,\pi/2)$.
\end{Rema}

By Proposition \ref{fourier.prop2}, the following dissipation rate of Fourier component $g_n$ is satisfies
\begin{equation}\label{eq32}
\|g_n(T,\cdot)\|_{\cH_{n}}\le e^{-|n|(T-t)}\|g_n(t,\cdot)\|_{\cH_{n}},\qquad\qquad\forall t\in (0,T).
\end{equation}
\begin{Nota}
	In what follows, to simplify the notation, we shall assume $n\in\NN$. The same considerations hold for $n\in\ZZ_{-}$ by replacing $n$ with $|n|$.
\end{Nota}

\subsection{Strategy for the proof of Theorem \ref{Theo2} and uniform observability}\label{ss.Observability}

We show in this subsection how the proof of Theorem \ref{Theo2} reduces to the proof of an observability inequality for the $1D$ parabolic equations \eqref{eq23} that is uniform with respect to $n\in\NN$.
Recall that if $g$ is the solution of \eqref{eq12}, then it can be represented by 
\begin{equation}\label{eq44}
g(t,x,y) = \sum\limits_{n\in\ZZ}g_n(t,x)e^{iny}, \qquad
\text{ for a.e. }(t,x,y)\in(0,T)\times\Omega.
\end{equation}
We also emphasize that, by Bessel-Parseval's equality, one has, for a.e. $t\in(0,T)$
and every $-\pi/2\le a_1\le b_1\le\pi/2$, that
\begin{equation}\label{eq45}
\int_{a_1}^{b_1}\int_{0}^{2\pi}|g(t,x,y)|^2d\sigma=\sum\limits_{n\in\ZZ}\int_{a_1}^{b_1}|g_n(t,x)|^2\cos xdx.
\end{equation}
Thus, if there exists a positive constant $C>0$, independent of $n\in\NN$, and such that the following uniform observability holds true for the system \eqref{eq23}
\begin{equation}\label{eq46}
\int_{-\frac\pi 2}^{\frac\pi 2}|g_n(T,x)|^2dx\le C\int_{0}^{T}\int_{\omega_{a,b}}|g_n(t,x)|^{2}\cos xdxdt,
\end{equation}
then, we can easily show that the observability inequality \eqref{eq13} is verified. Indeed, thanks to \eqref{eq44}, \eqref{eq45} and \eqref{eq46}, we find
\begin{multline}\label{eq47}
\int_{\Omega}|g(T,x,y)|^2d\sigma =\sum\limits_{|n|\le\ell\in\NN}\int_{-\frac\pi 2}^{\frac\pi 2}|g_n(T,x)|^{2}\cos xdx\\
\le C\sum\limits_{|n|\le \ell\in\NN}\int_{0}^{T}\int_{\omega_{a,b}}|g_n(t,x)|^{2}\cos xdxdt
=C\int_{0}^{T}\int_{\omega_{a,b}}\int_{0}^{2\pi}|g(t,x,y)|^{2}d\sigma dt.
\end{multline}
This immediately yields \eqref{eq13}. Hence, in order to prove Theorem \ref{Theo2} it is necessary and sufficient to study the observability of system \eqref{eq23} uniformly with respect to $n\in\NN$.

\begin{Defi}(Uniform observability)\label{fourier.def2}
	Let $\omega_{a,b}$ be defined as in \eqref{eq::omega}. Then system \eqref{eq23} is observable in $\omega_{a,b}$ in time $T$ uniformly with respect to $n\in\NN$, if there exists $C>0$ such that, for every $n\in\NN$, and $g_{0,n}\in\cH_{n}$, the solution of \eqref{eq23} satisfies \eqref{eq46}.
\end{Defi}

\section{The $1D$ equations in the space $ L^2$ without weight}\label{s::weight}

In this section, we recast the $1D$ equation \eqref{eq23} in the space $ L^2(-\pi/2,\pi/2)$ without weight in the cases $n\in \NN\bs\{0\}$ and in the space $ L^2(-1,1)$ when $n=0$.

\subsection{The $1D$ equation in the space $ L^2(-1,1)$ and observability inequality when $n=0$}\label{ss.Well4}
Let us consider the unitary transformation 
\begin{eqnarray*}
	\operatorname{V}:& L^2((-\pi/2,\pi/2);\cos xdx)&\longrightarrow  L^2(-1,1)\\
	&v&\longmapsto (\operatorname{V}v)(x) = v(\arcsin x).
\end{eqnarray*} 
We define the unbounded operator $\operatorname{M_0}$ on the space $ L^2(-1,1)$ by
\begin{equation}\label{eq::well1}
\operatorname{M_0} = \operatorname{V}\cL_0\operatorname{V}^{*},\qquad\qquad\operatorname{D(M_0)} = \operatorname{V}\operatorname{D}(\cL_0).
\end{equation} 
Here, $\operatorname{V}^{*}$ is the adjoint of the unitary operator $\operatorname{V}$, that is,
\begin{eqnarray*}
	\operatorname{V}^{*}:& L^2(-1,1)&\longrightarrow  L^2((-\pi/2,\pi/2);\cos xdx)\\
	&w&\longmapsto (\operatorname{V}w)(x) = w(\sin x).
\end{eqnarray*} 

We then have the following expression of operator $\operatorname{M_0}$:
\begin{equation}\label{eq::well2}
\operatorname{M_0}w =((1-x^2)w')',\qquad\qquad\forall w\in \operatorname{D(M_0)}.
\end{equation} 

Since the differential operator $\partial_t$ commutes with the unitary transformation $\operatorname{V}$, one deduces easily that, when $n=0$, system \eqref{eq23} is equivalent to the following 
\begin{equation}\label{eq::well3}
\begin{cases}
\partial_t\tilde{g}_0-\operatorname{M_0}\tilde{g}_0 = 0,&\mbox{ a.e. in } (0,T)\times(-1,1),\cr
\tilde{g}_0(0,x) = \tilde{g}_{0,0}(x),&\;x\in (-1,1).
\end{cases}
\end{equation}
In particular, the solution $\tilde{g}_0 = \operatorname{V}g_0$ lies in the class (see Proposition \ref{fourier.prop2} and Remark \ref{fourier.rmq2})
\begin{equation}\label{eq::well4}
\cC([0,T]; L^2(-1,1))\cap\cC^{\infty}((0,T);\operatorname{D(M_{0})}).
\end{equation}

We characterise in the following some useful properties of functions belonging to the domain $\operatorname{D(M_{0})}$, that are obtained by Lemma~\ref{lem::fourier 1}.

\begin{Lem}\label{lem::well1}
	Let $w\in\operatorname{D(M_{0})}$. Then $w$  belongs to the Sobolev space $ H^1(-1,1)$ and $w$ is locally absolutely continuous on $[-1,1]$. Moreover, it holds
	\begin{equation}\label{eq::well5}
	w(\pm 1)\in\RR\qquad\mbox{and}\qquad w'(x)\sqrt{1-x^2}|_{x=\pm 1} = 0.
	\end{equation}
\end{Lem}
\begin{proof}
	Let $w\in\operatorname{D(M_0)}$. Then $w(x) = v(\arcsin x)$ for some $v\in\operatorname{D(\cL_0)}$ and $\operatorname{a.e.}$, $x\in(-1,1)$. Since $v(\pm\pi/2)\in\RR$, the first property in \eqref{eq::well5} immediately follows. Similarly, since $v \in\cH_0\cong L^2((-\pi/2,\pi/2);\cos xdx)$, and using \eqref{eq::second derivative}, it holds
	$$
	\int_{-1}^{1}|w(x)|^2dx = \int_{-\frac\pi 2}^{\frac\pi 2}|v(x)|^2\cos xdx<\infty\qquad\mbox{and}\qquad\int_{-1}^{1}|w'(x)|^2dx = \int_{-\frac\pi 2}^{\frac\pi 2}\frac{|v'(x)|^2}{\cos x}dx<\infty,
	$$
	so that, $w, w'\in L^2(-1,1)$. Finally, $w'(x)\sqrt{1-x^2}|_{x=\pm 1} = v'(\pm\pi/2) = 0$, by \eqref{eq33}.
\end{proof}
\begin{Rema}
	We note that an observability inequality for equation \eqref{eq::well3} was established in \cite{Martinez} by Martinez and Vancostenoble. Indeed thanks to Lemma \ref{lem::well1} we aim at proving an observability inequality for the following equation
	\begin{equation}\label{eq::well6}
	\begin{cases}
	\partial_t w-\dx(a(x)\dx w) = 0,&\mbox{ a.e. in } (0,T)\times(-1,1),\cr
	(a(x)\dx w)(t,\pm 1) = 0,&\;t\in(0,T),\cr
	w(0,x) = w_0(x),&\;x\in (-1,1),
	\end{cases}
	\end{equation}
	where $a(x):= 1-x^2$, $w_0\in L^2(-1,1)$ and the solution $w$ belongs to the class \eqref{eq::well4}. We observe that the weight function $a$ satisfies $0\le a\in\cC^2([-1,1]),\; a(\pm 1) = 0$, $a>0$ on $(-1,1)$, $\frac{1}{\sqrt{a}}\in L^1(-1,1)$  and
	$$
	\frac{(1+x)a'(x)}{a(x)}\xrightarrow[x\to-1^{+}]{}1\qquad\mbox{and}\qquad\frac{(1-x)a'(x)}{a(x)}\xrightarrow[x\to1^{-}]{}-1.
	$$
	So, we are in the framework of \cite{Martinez}.
\end{Rema}
Then \cite[Theorem 3.4]{Martinez} gives the following
\begin{Lem}
	Let $T>0$ and $a,b\in\RR$ be such that $0<a<b\le\frac\pi 2$. Let $\widetilde{\omega}_{a,b}:=(-\sin b,-\sin a)\cup (\sin a,\sin b)$. Then, there exists a positive constant $C_0>0$ such that every solution $w$ of system \eqref{eq::well6} satisfies
	\begin{equation}
	\int_{-1}^{1}|w(T,x)|^2dx\le C_0\int_{0}^{T}\int_{\widetilde{\omega}_{a,b}}|w(t,x)|^2dxdt.
	\end{equation}	
\end{Lem}
Finally, thanks to the above lemma and the fact that $w:=\operatorname{V}g_0$ is the solution of system \eqref{eq::well6}, we deduce the following observability inequality for equation \eqref{eq23} when $n=0$. 
\begin{Prop}
	Let $T>0$ and $\omega_{a,b}$ be defined as in \eqref{eq::omega}. Then, there exists a positive constant $C_0>0$ such that the first Fourier component $g_0$, which is the solution of equation \eqref{eq23} when $n=0$ satisfies
	\begin{equation}\label{eq::inequality 0}
	\int_{-\frac\pi 2}^{\frac\pi 2}|g_0(T,x)|^2\cos xdx\le C_0\int_{0}^{T}\int_{\omega_{a,b}}|g_0(t,x)|^2\cos xdxdt.
	\end{equation}
\end{Prop}
\begin{Rema}\label{rmk::usefule} We highlight that the result of \cite[Theorem 3.4]{Martinez} ensures that when $n=0$, system \eqref{eq23} is observable in any subset $\omega\subset\subset(-\pi/2,\pi/2)$ and in arbitrary time $T>0$.
\end{Rema}
\subsection{The $1D$ equations in the space $ L^2(-\pi/2,\pi/2)$ without weight in the cases $n\in\NN^{*}$}\label{ss.Well3}
In these cases, we consider the unitary transformation 
\begin{eqnarray*}
	\operatorname{U}:& L^2((-\pi/2,\pi/2);\cos xdx)&\longrightarrow  L^2(-\pi/2,\pi/2)\\
	&v&\longmapsto (\operatorname{U}v)(x) = \sqrt{\cos x}v(x).
\end{eqnarray*} 
We define for all $n\in\NN^{*}$ the unbounded operator $\operatorname{M_n}$ on the space $ L^2(-\pi/2,\pi/2)$ by
\begin{equation}\label{eq35}
\operatorname{M_n} = \operatorname{U}\cL_{n}\operatorname{U}^{*},\qquad\qquad\operatorname{\operatorname{D(M_n)}} = \operatorname{U}\operatorname{D}(\cL_{n}),
\end{equation} 
where, $\operatorname{U}^{*}$ is the adjoint of the unitary operator $\operatorname{U}$. So, we deduce the following expression of operator $\operatorname{M_n}$:
\begin{equation}\label{eq36}
\operatorname{M_n}w = w''-q_{n}(x)w,\hspace{1cm}\forall w\in \operatorname{D(M_n)},
\end{equation} 
where, for all $n\in\NN^{*}$, the potential $q_n$ is given by
\begin{equation}\label{eq37}
q_n(x) = (n^2-1/4)\tan^2x-1/2,\hspace{1cm}\forall x\in(-\pi/2,\pi/2).
\end{equation}
\begin{Rema}\label{rmq1}
	Let us emphasis that, since $\operatorname{U}$ is an unitary transformation, then the unbounded operator $(\operatorname{M_n},\operatorname{D(M_n)})$ defined on the space $ L^2(-\pi/2,\pi/2)$ inherits some properties of the operator $(\cL_{n},\operatorname{D}(\cL_{n}))$. That is, the operator $(-\operatorname{M_n},\operatorname{D(M_n)})$ is a densely defined, self-adjoint, and positive operator with compact resolvent on $L^2(-\pi/2,\pi/2)$ for all $n\in\NN^{*}$. 
	
\end{Rema}

Since the differential operator $\partial_t$ commutes with the unitary transformation $\operatorname{U}$, one deduces easily that system \eqref{eq23} is equivalent to the following 
\begin{equation}\label{eq39}
\begin{cases}
\partial_t\tilde{g}_n-\operatorname{M_n}\tilde{g}_n = 0,&\mbox{ a.e. in } (0,T)\times(-\pi/2,\pi/2),\cr
\tilde{g}_n(0,x) = \tilde{g}_{0,n}(x),&\;x\in (-\pi/2,\pi/2).
\end{cases}
\end{equation}
In particular, the solution $\tilde{g}_n = \operatorname{U}g_n$ lies in the class (see Proposition \ref{fourier.prop2} and Remark \ref{fourier.rmq2})
\begin{equation}\label{eq40}
\cC([0,T];L^2(-\pi/2,\pi/2))\cap\cC^{\infty}((0,T);\operatorname{D(M_{n})}).
\end{equation}
In the following, we collect some properties of the functions lying in the domain $\operatorname{D(M_n)}$ ($n\in\NN^{*}$) which will be useful in the proof of a global Carleman estimate for system \eqref{eq39} in Section \ref{s.Carleman}. 
\begin{Lem}\label{Lem::useful}
	Let $n\in\NN^{*}$ and $w\in\operatorname{D(M_n)}$. Then $w'$ belongs to $ L^2(-\pi/2,\pi/2) $ and $w$ is locally absolutely continuous on $\displaystyle[-\pi/2,\pi/2]$. Moreover,
		\begin{equation}\label{eq41}
	\lim\limits_{x\rightarrow-\frac{\pi}{2}^{+}}w(x)=\lim\limits_{x\rightarrow\frac{\pi}{2}^{-}}w(x)=0,\qquad\qquad\lim\limits_{x\rightarrow-\frac{\pi}{2}^{+}}w'(x)=\lim\limits_{x\rightarrow\frac{\pi}{2}^{-}}w'(x)=0.
	\end{equation}
\end{Lem}
\begin{proof}
	Let $n\in\NN^{*}$ and $w\in\operatorname{D(M_n)}$. Then $w = \sqrt{\cos x}v$ for some $v\in\operatorname{D(\cL_n)}$ and $\operatorname{a.e.}$, $x\in\displaystyle(-\pi/2,\pi/2)$. Since $v(\pm\pi/2)\in\RR$, the first identity in \eqref{eq41} immediately follows. Similarly, since $\tan xv\in\cH_n$ (see, \eqref{eq::A1}), it holds
	$$
	\|\tan xw\|_{ L^2(-\pi/2,\pi/2)}^2=\|\tan xv\|_{\cH_n}^2<\infty.
	$$
	By deriving $w$, we find that $w'+\tan xw/2 = \sqrt{\cos x}v'$ belongs to $ L^2(-\pi/2,\pi/2)$, due to \eqref{eq::A1}, and then $w'\in L^2(-\pi/2,\pi/2)$. Since $v'(\pm\pi/2)\in\RR$, it holds
	$$
	\lim\limits_{x\rightarrow-\frac{\pi}{2}^{+}}w'(x) = \lim\limits_{x\rightarrow-\frac{\pi}{2}^{+}}-\frac{1}{2}\frac{\sin x}{\sqrt{\cos x}}v(x)+\sqrt{\cos x}v'(x) = 0,
	$$
	by the second identity of \eqref{eq34}. The proof of the limit at $\pi/2$ is similar. It then follows that $w$ is locally absolutely continuous on $\displaystyle[-\pi/2,\pi/2]$. 
\end{proof}
\begin{Rema}\label{rmq::useful1}
	The above lemma also shows that for all $n\in\NN^{*}$, the domain $\operatorname{D(M_n)}$ is a subspace of the Sobolev space $\displaystyle H_0^1(-\pi/2,\pi/2)$. Therefore, Lemma~\ref{lem:Hardy-Poincaré} holds true in $\operatorname{D(M_n)}$.
\end{Rema}

\section{A global Carleman estimate in the cases $n\in\NN^{*}$}\label{s.Carleman}
The purpose of this section is to obtain a global Carleman estimate for systems \eqref{eq39} in the case $n\in\NN^{*}$. This will allow us, using the dissipation rate \eqref{eq32}, to prove the uniform observability inequality \eqref{eq46} in Section \ref{ss.proof_positive}. In what follows, we drop the tilde and the index $n$ to simplify the notations. 

\begin{Prop}\label{Carleman.pro}
	Let $\omega_{a,b}$ be defined as in \eqref{eq::omega}. Then there exist a weight function $\beta\in\cC^{4}([-\pi/2,\pi/2])$ and positive constants $\cR_0, \cR_1>0$ such that for every $T>0$, $n\in\NN^{*}$ and $s\ge\cR_0\max(T+T^2,T^2n)$, every $\displaystyle g\in \cC([0,T]; L^2(-\pi/2,\pi/2))\cap\cC^{2}((0,T);\operatorname{D(M_{n})})$ satisfies
	\begin{multline}\label{eq48}
	\cR_1\int_{0}^{T}\int_{-\frac\pi 2}^{\frac\pi 2}\left(\frac{s}{t(T-t)}|\dx g(t,x)|^2+\frac{s^3}{(t(T-t))^3}|g(t,x)|^2\right)e^{-\frac{2s\beta(x)}{t(T-t)}}dxdt \\
	\le\int_{0}^{T}\int_{\omega_{a,b}}\frac{s^3}{(t(T-t))^3}|g(t,x)|^2e^{-\frac{2s\beta(x)}{t(T-t)}}dxdt
	+\int_{0}^{T}\int_{-\frac\pi 2}^{\frac\pi 2}|\cP_ng(t,x)|^2e^{-\frac{2s\beta(x)}{t(T-t)}}dxdt.
	\end{multline}
	Here, $\cR_i:=\cR_i(\beta,a,b)$, $i=0,1$ and we let
	$$\cP_n:=\dt-\dx^2+q_n(x)\hspace{0.5cm}\mbox{with}\hspace{0.5cm}q_n(x) = (n^2-1/4)\tan^2x-1/2.$$
\end{Prop}

Before proving the above proposition, let us present some important remarks and comments which are essential to understand the proof.

As it is now well-understood, the main difficulty in the proof of Carleman estimates as \eqref{eq48} is to identify a suitable weight function $\beta$ which is able to deal with the specificity of the parabolic operator under consideration. 
For example, for the standard parabolic operator see the pioneer work by Imanuvilov \cite{Imanuvilov} or Fursikov and Imanuvilov \cite{Fursikov}; for the standard parabolic operator with interior quadratic singularities (resp.~boundary singularity) see the work by Ervedoza \cite{Ervedoza} (resp.~Cazacu \cite{Cazacu} or Biccari and Zuazua \cite{Biccari}); for the $2D$ parabolic Grushin operator, see the work by Beauchard and al \cite{Beau,Beau2,Beau4} and Koenig \cite{Koenig}; for $2D$ parabolic Grushin operator with internal (resp. boundary) singular potential see the work by Morancey \cite{Morancey} (resp. Cannarsa and Guglielmi \cite{Cannarsa_G}). We remark that in general, the function $\beta$ is chosen to be strictly monotone outside of the control region, and concave, so that the term in $s^3$ is the leading one. 
Particularly in the singular cases, this choice allows to get rid of the singular terms which can not be bounded at the singularity, usually by taking advantage of Hardy-Poincaré type inequalities.

In the case at hand, the potential $q_n$ is singular in $\pm\pi/2$. Thus, we shall apply the Hardy-Poincaré inequalities of Lemma~\ref{lem:Hardy-Poincaré} (see, Remark~\ref{rmq::useful1}) to get rid of the singular terms which can not be bounded at $\pm\pi/2$.
\begin{Rema}\label{Carleman.rmq1}
	The proof of Proposition \ref{Carleman.pro} will be split into several lemmas using the classical strategy \cite{Fursikov} by Fursikov and Imanuvilov (we refer to \cite[p.79]{Coron} for a pedagogical presentation). Let us emphasize that functions in $\operatorname{D(M_{n})}$ have the regularity in the space variable that we need in order to apply integrations by parts (see, $\operatorname{e.g.}$ Lemma \ref{Lem::useful}).
\end{Rema}
\begin{Nota}\label{Carleman.nota1}
	Let us introduce the general notations which will be used in what follows. We let $a'$ and $b'$ are real numbers such that
	\begin{equation}\label{eq49}
	0<a<a'<b'<b\le\pi/2\qquad\mbox{and}\qquad [a',b']\subset(a,b).
	\end{equation}
	We consider the subdomains
	\begin{equation}\label{eq50}
	\operatorname{\omega}_{con}:=(-b',-a')\cup(a',b'),\qquad\operatorname{\omega}_{deg}:=(-a',a'),\qquad\operatorname{\omega}_{bdy}:=\left(-\pi/2,-b'\right)\cup\left(b',\pi/2\right),
	\end{equation}
	so that $$(-\pi/2,\pi/2) = \operatorname{\omega}_{bdy}\cup\operatorname{\omega}_{deg}\cup\operatorname{\omega}_{con}\qquad\qquad\mbox{and}\qquad\qquad \operatorname{\omega}_{con}\subset\omega_{a,b}.$$
	We introduce also the weight function
	\begin{equation}\label{eq51}
	\varphi(t,x) = s\theta(t)\beta(x),\qquad\qquad(t,x)\in Q:=(0,T)\times I,\qquad I:=(-\pi/2,\pi/2),
	\end{equation}
	where the positive constant $s = s(T,n,\beta)>0$ will be chosen later on and the temporal weight $\theta$ is given by
	\begin{equation}\label{eq52}
	\theta(t) = \frac{1}{t(T-t)},\qquad\qquad t\in(0,T).
	\end{equation}
	We end this part of notations introducing for all $n\in\NN^{*}$ and every $g\in \cC([0,T]; L^2(-\pi/2,\pi/2))\cap\cC^{2}((0,T);\operatorname{D(M_{n})})$,  the change of function
	\begin{equation}\label{eq53}
	z(t,x) = g(t,x)e^{-\varphi(t,x)},\hspace{1cm}(t,x)\in Q.
	\end{equation}
\end{Nota}
In the following lemma we design the weight function $\beta$.
\begin{Lem}\label{Carleman.lem1}
	The function $\beta\in\cC^{4}([-\pi/2,\pi/2])$ satisfies
	\begin{equation}\label{eq54}
	\beta\ge 1,\qquad\qquad\mbox{on}\qquad(-\pi/2,\pi/2),
	\end{equation}
	\begin{equation}\label{eq55}
	\beta(x)=	\begin{cases}
	\log|\sin x|+A_1|x|+A_2&\mbox{if}\qquad x\in\overline{\operatorname{\omega}_{bdy}},\cr
	\log\cos x-\frac{x^2}{2}+A_3(x+1)&\mbox{if}\qquad x\in\overline{\operatorname{\omega}_{deg}},
	\end{cases}
	\end{equation}
	where the positive constants $A_i,\;1\le i\le 3$ are such that \eqref{eq54} is verified and
	\begin{equation}\label{eq57}
	\begin{cases}
	|\beta'(x)|\ge\eta_1,& x\in\overline{\operatorname{\omega}_{bdy}},\cr
	\beta'(x)\ge\eta_2,& x\in\overline{\operatorname{\omega}_{deg}},
	\end{cases}
	\end{equation}
	for some positive constants $\eta_1,\eta_2>0$.
\end{Lem}
\begin{figure}[h!]
		\centering
		\includegraphics[width=0.45\textwidth]{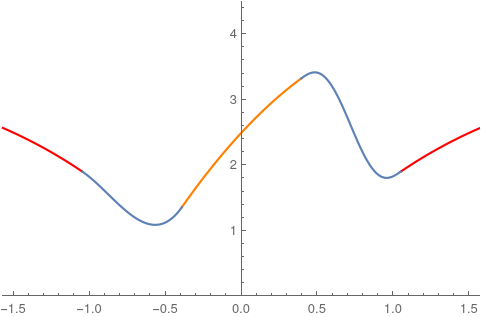}
		\caption{ The spatial weight function $\beta$. The subcontrol region $\omega_{a',b'}$ is in blue.}
		\label{fig2}
\end{figure}
\begin{Rema}\label{Carleman.rmq2}
	We stress that the explicit expression of the weight $\beta$ is only needed near $\pm\pi/2$, in order to get rid of the singular terms which can not be bounded at $\pm\pi/2$. Apart from this, assuming that $\beta$ is strictly monotonous and concave outside the subcontrol region $\operatorname{\omega}_{con}$ suffices.
\end{Rema}
The following lemma gives some useful properties of the temporal weight $\theta$ which are obtained by direct computations.
\begin{Lem}\label{Carleman.lem2}
	Let the temporal weight $\theta$ be given by \eqref{eq52}. Then we have for all $t\in (0,T)$,
	$$\theta'(t) = (2t-T)\theta^2(t),\qquad\theta''(t) = 2\theta^2(t)(1+(2t-T)^2\theta(t)),$$
	and the following inequalities hold
	$$\theta(t)\le 2^{-4}T^{4}\theta^{3}(t),\qquad|\theta'(t)|\le 2^{-2}T^{3}\theta^3(t),\qquad|\theta(t)\theta'(t)|\le T\theta^3(t),\qquad|\theta''(t)|\le\frac{5}{2}T^{2}\theta^3(t).$$
	Moreover, one has
	$$\lim\limits_{t\rightarrow 0^{+}}\theta(t)=\lim\limits_{t\rightarrow T^{-}}\theta(t)=+\infty.$$
\end{Lem}
In the following lemma, we give some useful properties of the function $z$ introduced in \eqref{eq53} which are obtained by direct computations  applying Lemmas \ref{Lem::useful} and \ref{Carleman.lem2}
\begin{Lem}\label{Carleman.lem3}
	Let $n\in\NN^{*}$, then the function $z$ introduced in \eqref{eq53} belongs at least in the class $\cC([0,T]; L^2(-\pi/2,\pi/2))\cap\cC^{2}((0,T);\operatorname{D(M_{n})})$ and satisfies
	\begin{equation}\label{eq58}
	\begin{cases}
	z(0,x) = z(T,x) = \partial_{x}z(0,x) = \partial_{x}z(T,x) = 0,&x\in [-\pi/2,\pi/2],\cr
	z(t,\pm\pi/2) = \partial_{t}z(t,\pm\pi/2) = \dx z(t,\pm\pi/2) = 0,&t\in(0,T).
	\end{cases}
	\end{equation}
	Moreover, one has
	\begin{equation}\label{eq59}
	\cP_n^{+}z+\cP_n^{-}z = e^{-\varphi}\cP_ng,
	\end{equation}
	where $\cP_n$ is the parabolic operator introduced in Proposition \ref{Carleman.pro}, and we let
	\begin{equation}\label{eq60}
	\cP_n^{+}z = -\operatorname{M_n}z+(\partial_{t}\varphi-|\dx\varphi|^2)z\qquad\mbox{and}\qquad\cP_n^{-}z = \partial_{t}z-2\dx z\dx\varphi-(\dx^2\varphi)z.
	\end{equation}
\end{Lem}

Let $Q=(0,T)\times (-\pi/2,\pi/2)$ and $dQ = dxdt$. Observe first that, $\cP_n^{+}z$ and $\cP_n^{-}z$ belong to $L^2(Q)$ by the definition of $\operatorname{D(M_{n})}$ and Lemma \ref{Lem::useful}. So, developing the $L^2(Q)$ squared norm in identity \eqref{eq59},  leads to
\begin{equation}\label{eq61}
\int_{Q}\cP_n^{+}z\cP_n^{-}zdQ\le\frac{1}{2}\int_{Q}\left|e^{-\varphi}\cP_ng\right|^{2}dQ.
\end{equation}

In the following we compute the scalar product in the left hand side of \eqref{eq61} using integration by parts and Fubini's Theorem.
\begin{Lem}\label{Carleman.lem4}
	Let $n\in\NN^{*}$, then we have
	\begin{multline}\label{eq62}
	\int_{Q}\cP_n^{+}z\cP_n^{-}zdQ = -2\int_{Q}\dx^2\varphi|\dx z|^2dQ+\frac{1}{2}\int_{Q}\dx^4\varphi|z|^2dQ+\int_{Q}\dx\varphi q_n'(x)|z|^2dQ\\
	-\frac{1}{2}\int_{Q}(\partial_{t}^2\varphi-2\dx\varphi\partial_{tx}\varphi)|z|^2dQ+\int_{Q}\dx\varphi\dx(\partial_{t}\varphi-|\dx\varphi|^2)|z|^2dQ.
	\end{multline}
\end{Lem}
\begin{proof}
	Let $n\in\NN^{*}$ and $z$ be defined in \eqref{eq53}.
	We compute the six terms of the left hand side of \eqref{eq62} using integration by parts and Fubini's Theorem.
	
	1. The three terms involving operator $-\operatorname{M_n}z$ are
	$$-\int_{Q}\operatorname{M_n}z\dt zdQ+2\int_{Q}\operatorname{M_n}z\dx z\dx\varphi dQ+\int_{Q}\operatorname{M_n}z\dx^2\varphi zdQ.$$
	So, one has using \eqref{eq58},
	\begin{eqnarray}\label{eq63}
	-\int_{Q}\operatorname{M_n}z\partial_{t}zdQ &=& \int_{0}^{T}\left\{\left[-\dx z\dx(\dt z)\right]_{-\frac{\pi}{2}}^{\frac{\pi}{2}}+\int_{-\frac\pi 2}^{\frac\pi 2}\frac{1}{2}\dt(|\dx z|^2)dx\right\}dt+ \int_{-\frac\pi 2}^{\frac\pi 2}\frac{q_n(x)}{2}\left[|z|^2\right]_{0}^{T}dx\nonumber\\
	&=&\frac{1}{2}\int_{-\frac\pi 2}^{\frac\pi 2}\left\{|\partial_{x}z(T,x)|^{2}-|\partial_{x}z(0,x)|^{2}\right\}dx=0.
	\end{eqnarray}
	\begin{eqnarray}\label{eq64}
	2\int_{Q}\operatorname{M_n}z\dx z\dx\varphi dQ &=& -\int_{Q}\dx^2\varphi|\dx z|^2dQ+\int_{Q}\dx(q_n(x)\dx\varphi)|z|^2dQ\nonumber\\
	&&+\int_{0}^{T}\left\{[|\dx z|^2\dx\varphi]_{-\frac{\pi}{2}}^{\frac{\pi}{2}}-[q_n(x)\dx\varphi|z|^2]_{-\frac{\pi}{2}}^{\frac{\pi}{2}}\right\}dt\nonumber\\
	&=&-\int_{Q}\dx^2\varphi|\dx z|^2dQ+\int_{Q}\dx(q_n(x)\dx\varphi)|z|^2dQ.
	\end{eqnarray}
	Observe that boundary terms $\displaystyle[|\dx z|^2\dx\varphi]_{-\frac{\pi}{2}}^{\frac{\pi}{2}} = 0$ and $\displaystyle[q_n(x)\dx\varphi|z|^2]_{-\frac{\pi}{2}}^{\frac{\pi}{2}}=0$ due to \eqref{eq41} and \eqref{eq34}. Remember that $z(t,\cdot)\in \operatorname{D(M_{n})}$ implies $z(t,\cdot) = \sqrt{\cos x}v(t,\cdot)$ for some $v(t,\cdot)\in \operatorname{D}(\cL_{n})$.
	
	The third term is
	\begin{eqnarray}\label{eq65}
	\int_{Q}\operatorname{M_n}z\dx^2\varphi zdQ&=&-\int_{Q}\dx^2\varphi|\dx z|^2dQ+\frac{1}{2}\int_{Q}\dx^4\varphi|z|^2dQ-\int_{Q}q_n(x)\dx^2\varphi|z|^2dQ.
	\end{eqnarray}

	2.  The three terms involving $\left(\partial_{t}\varphi-|\dx\varphi|^2\right)z$ are
	$$\int_{Q}(\partial_{t}\varphi-|\dx\varphi|^2)z\partial_{t}zdQ-2\int_{Q}(\partial_{t}\varphi-|\dx\varphi|^2)z\dx z\dx\varphi dQ-\int_{Q}(\partial_{t}\varphi-|\dx\varphi|^2)\dx^2\varphi|z|^2dQ.$$
	So one has,
	\begin{equation}\label{eq69}
	\int_{Q}(\partial_{t}\varphi-|\dx\varphi|^2)z\partial_{t}zdQ=-\frac{1}{2}\int_{Q}\partial_{t}(\partial_{t}\varphi-|\dx\varphi|^2)|z|^2dQ,
	\end{equation}
	Here, the boundary term vanish as $t\rightarrow 0^{+},T^{-}$, because owing to \eqref{eq53}, \eqref{eq54} and Lemma \ref{Carleman.lem2} one has
	$$|(\partial_{t}\varphi-|\dx\varphi|^2)|z|^2|\le \theta^2e^{-s\theta}|s|2t-T|\beta+(s\theta\beta')^2||g|^2,$$
	and the right hand side tends to zero as $t\rightarrow 0^{+},T^{-}$ for every $x\in[-\pi/2,\pi/2]$. Using \eqref{eq58} one get
	\begin{equation}\label{eq70}
	-2\int_{Q}(\partial_{t}\varphi-|\dx\varphi|^2)z\dx z\dx\varphi dQ=\int_{Q}\dx[(\partial_{t}\varphi-|\dx\varphi|^2)\dx\varphi]|z|^2dQ.
	\end{equation}
	Finally, the third term is just
	\begin{equation}\label{eq71}
	\hspace{-5cm}-\int_{Q}(\partial_{t}\varphi-|\dx\varphi|^2)\dx^2\varphi|z|^2dQ.
	\end{equation}
	
	By combining \eqref{eq63}-\eqref{eq71}, we complete the proof of the lemma.
\end{proof}
We are going now to bound from below the right hand side of \eqref{eq62}. Since $Q:=(0,T)\times(-\pi/2,\pi/2)=(0,T)\times(\operatorname{\omega}_{bdy}\cup\operatorname{\omega}_{con}\cup\operatorname{\omega}_{deg})$, we separate the integrals of the right hand side over $(0,T)\times J$, where $J\in\{\operatorname{\omega}_{bdy},\operatorname{\omega}_{con},\operatorname{\omega}_{deg}\}$. Using Lemma \ref{Carleman.lem1} we immediately get
\begin{Lem}\label{Carleman.lem5}
	Let $n\in\NN^{*}$ and assume \eqref{eq62}. Then one has
	\begin{equation}\label{eq72}
	\int_{Q}\cP_n^{+}z\cP_n^{-}zdQ = \int_{0}^{T}\int_{\operatorname{\omega}_{bdy}}\operatorname{K_{bdy}}dQ+\int_{0}^{T}\int_{\operatorname{\omega}_{con}}\operatorname{K_{con}}dQ+\int_{0}^{T}\int_{\operatorname{\omega}_{deg}}\operatorname{K_{deg}}dQ,
	\end{equation}
	where
	\begin{eqnarray}\label{eq73}
	\operatorname{K_{deg}}&=&s\theta\left\{\left(\frac{2}{\cos^2x}+2\right)|\dx z|^2+\left(\frac{\sin^2x}{2\cos^4x}+\frac{x\sin x}{2\cos^3x}\right)|z|^2\right\}\nonumber\\
	&&+ \left\{\frac{2n^2s\theta}{\cos^2x}+2s^3\theta^3\left(\frac{1}{\cos^2x}+1\right)(-\tan x-x+A_3)^2\right\}|z|^2\nonumber\\
	&&+s\theta\left\{ A_3\frac{(2n^2-1/2)}{\cos^4x}+2s\theta'(-\tan x-x+A_3)^2- 2n^2\frac{x\tan x}{\cos^2x}\right\}|z|^2\nonumber\\
	&&-\frac{s\theta''}{2}(\log\cos x-\frac{x^2}{2}+A_3(x+1))|z|^2,
	\end{eqnarray}
	
	\begin{eqnarray}\label{eq74}
	\operatorname{K_{bdy}}&=&\frac{2s\theta}{\sin^2x}\left\{|\dx z|^2+|z|^2\right\}+\frac{2s^3\theta^3}{\sin^2x}\left(\frac{\cos x}{\sin x}+A_1\operatorname{sign}(x)\right)^2|z|^2\nonumber\\
	&&+\left\{2s^2\theta\theta'\left(\frac{\cos x}{\sin x}+A_1\operatorname{sign}(x)\right)^2-\frac{s\theta''}{2}(\log|\sin x|+A_1|x|+A_2)\right\}|z|^2\nonumber\\
	&&+\left\{-\frac{3s\theta}{\sin^4x}+\frac{s\theta(2n^2-1/2)}{\cos^2x}(1+A_1\operatorname{sign}(x)\tan x)\right\}|z|^2,
	\end{eqnarray}
	and
	\begin{equation}\label{eq76}
	\operatorname{K_{con}}=s\left\{\frac{\theta\beta^{(4)}}{2}+\theta\beta'\left(2n^2-\frac{1}{2}\right)\frac{\sin x}{\cos^3x}-\frac {\theta''\beta}2+2s\theta\beta'^2(\theta'-s\theta^2\beta'')\right\}|z|^2-2s\theta\beta''|\dx z|^2.
	\end{equation}
\end{Lem}
In the following lemmas, we bound from below \eqref{eq73} and \eqref{eq74} by positive terms.
\begin{Lem}\label{Carleman.lem6}
	Let $n\in\NN^{*}$ and $\operatorname{K_{deg}}$ be given by \eqref{eq73}. Then there exists a positive constant $s_1>0$ such that, for all
	\begin{equation}\label{eq77}
	s\ge s_1\max(T+T^2,T^2n),
	\end{equation}
	the following inequality holds
	\begin{equation}\label{eq78}
	\int_{0}^{T}\int_{\operatorname{\omega}_{deg}}\operatorname{K_{deg}}dQ\ge \int_{0}^{T}\int_{\operatorname{\omega}_{deg}}4s\theta|\dx z|^2+\eta_2^2s^3\theta^3|z|^2dQ,
	\end{equation}
	with $\eta_2$ as in \eqref{eq57}.
\end{Lem}
\begin{proof}
	Let $n\in\NN^{*}$ and $\operatorname{K_{deg}}$ be given by \eqref{eq73}. Since $x\sin x\ge 0$ for all $x\in[-\pi,\pi]$, we obtain using \eqref{eq57} that
	\begin{eqnarray}\label{eq79}
	\operatorname{K_{deg}}&\ge&4s\theta|\dx z|^2+4\eta_2^2s^3\theta^3|z|^2-\frac{s\theta''}{2}(\log\cos x-\frac{x^2}{2}+A_3(x+1))|z|^2\nonumber\\
	&&+\left\{2s^2\theta\theta'(-\tan x-x+A_3)^2-2s\theta n^2\frac{x\tan x}{\cos^2x}\right\}|z|^2.
	\end{eqnarray}
	For all $x\in\operatorname{\omega}_{deg}$ we have owing to Lemma \ref{Carleman.lem2}
	$$\left|2s^2\theta\theta'(-\tan x-x+A_3)^2\right|\le2C_1s^2T\theta^3,\qquad\left|-\frac{s\theta''}{2}(\log\cos x-\frac{x^2}{2}+A_3(x+1))\right|\le\frac{5C_2}{4}sT^2\theta^3,$$
	where $C_1=C_1(a'):=|\tan a'+a'+A_3|^2$ and $C_2=C_2(a'):=|\log\cos a'-a'^2/2+A_3(a'+1)|.$
	It follows that, if $
	s\ge \max\left(\frac{2C_1}{\eta_2^2},\frac{\sqrt{5C_2}}{2\eta_2}\right)T,$
	then \eqref{eq79} yields to
	\begin{equation}\label{eq81}
	\operatorname{K_{deg}}\ge4s\theta|\dx z|^2+2\eta_2^2s^3\theta^3|z|^2-2s\theta n^2\frac{x\tan x}{\cos^2x}|z|^2.
	\end{equation}

	Due to Lemma \ref{Carleman.lem2}, one has for all $x\in\operatorname{\omega}_{deg}$,
	$$\left|-2s\theta n^2\frac{x\tan x}{\cos^2x}\right|\le2n^2C_4s2^{-4}T^4\theta^3,$$
	where $C_4=C_4(a'):= a'\tan a'/\cos^2a'$.
	So, if
	\begin{equation}\label{eq86}
	s\ge\max\left(\frac{2C_1}{\eta_2^2},\frac{\sqrt{5C_2}}{2\eta_2},\frac{1}{\eta_2}\sqrt{\frac{C_4}{8}}\right)\max(T,T^2n),
	\end{equation}
	then
	\begin{equation}\label{eq87}
	\operatorname{K_{deg}}\ge4s\theta|\dx z|^2+\eta_2^2s^3\theta^3|z|^2.
	\end{equation}
	Therefore, owing to \eqref{eq86}-\eqref{eq87}, we deduce that, for all $n\in\NN^{*}$, if $s\ge s_1\max(T+T^2,T^2n)$,
	with
	\begin{equation}\label{eq88}
	s_1=s_1(a'):= \max\left(\frac{2C_1}{\eta_2^2},\frac{\sqrt{5C_2}}{2\eta_2},\frac{1}{\eta_2}\sqrt{\frac{A_5C_3}{32}},\frac{1}{\eta_2}\sqrt{\frac{C_4}{8}}\right),
	\end{equation}
	then
	\begin{equation}\label{eq89}
	\operatorname{K_{deg}}\ge 4s\theta|\dx z|^2+\eta_2^2s^3\theta^3|z|^2.
	\end{equation}
	This completes the proof of lemma.
\end{proof}
\begin{Lem}\label{Carleman.lem7}
	Let $n\in\NN^{*}$. Then there exists a positive constant $s_2>0$ such that, for all
	\begin{equation}\label{eq90}
	s\ge s_2(T+T^2),
	\end{equation}
	the following inequality holds
	\begin{equation}\label{eq91}
	\int_{0}^{T}\int_{\operatorname{\omega}_{bdy}}\operatorname{K_{bdy}}dQ\ge \int_{0}^{T}\int_{\operatorname{\omega}_{bdy}}2s\theta|\dx z|^2+2s\theta|z|^2+\frac{\eta_1^2}{2}s^3\theta^3|z|^2dQ,
	\end{equation}
	with $\eta_1$ as in \eqref{eq57} and $\operatorname{K_{bdy}}$ be given by \eqref{eq74}. 
\end{Lem}
\begin{proof}
	Let $n\in\NN^{*}$ and $\operatorname{K_{bdy}}$ be given by \eqref{eq74}. 
	Due to Lemma \ref{Carleman.lem2}, one has for all $x\in\operatorname{\omega}_{bdy}$,
	$$\left|-3s\theta/\sin^4x\right|\le3C_5s2^{-4}T^4\theta^3,\qquad\left|2s^2\theta\theta'\left(\cos x/\sin x+A_1\operatorname{sign}(x)\right)^2\right|\le2C_6s^2T\theta^3,$$
	and
	$$\left|-(s\theta''/2)(\log|\sin x|+A_1|x|+A_2)\right|\le 5\times 2^{-2}C_7sT^2\theta^3.$$
	Here, $C_5=C_5(b'):= 1/\sin^4b$, $C_6=C_6(b'):=\left(\cos b'/\sin b'+A_1\operatorname{sign}(b')\right)^2$ and $C_7:= A_1\left(\frac{\pi}2+1\right).$
	So, if $s\ge s_2(T+T^2)$,
	with $s_2=s_2(b'):=\max\left(\frac{\sqrt{3C_5}}{2\eta_1},\frac{\sqrt{5C_7}}{\eta_1},\frac{8C_6}{\eta_1^2}\right)$,
	then
	\begin{equation}\label{eq92}
\operatorname{K_{bdy}}\ge 2s\theta|\dx z|^2+2s\theta|z|^2+\frac{\eta_1^2}{2}s^3\theta^3|z|^2+\frac{s\theta(2n^2-1/2)}{\cos^2x}(1+A_1\operatorname{sign}(x)\tan x)|z|^2.
	\end{equation}
	Observe now that, since $0\le\operatorname{sign}(x)\sin x\le 1$ for all $x\in\operatorname{\omega}_{bdy}:=(-\pi/2,-b')\cup(b',\pi/2)$, we obtain
	\begin{eqnarray}
	0<\int_{0}^{T}\int_{\operatorname{\omega}_{bdy}}\hspace{-0.4cm}\frac{(1+A_1\operatorname{sign}(x)\tan x)}{\cos^2x}|z(t,x)|^2dQ\hspace{-0.2cm} &\le&\hspace{-0.2cm} \int_{0}^{T}\int_{-\frac{\pi}{2}}^{\frac{\pi}{2}}\frac{|z(t,x)|^2}{\cos^2x}dQ+A_1\int_{0}^{T}\int_{-\frac{\pi}{2}}^{\frac{\pi}{2}}\frac{|v(t,x)|^2}{\cos^2x}dQ\nonumber\\
	&<&\infty,
	\end{eqnarray}
	by Hardy-Poincaré inequalities \eqref{eq::Poncaré-Hardy} (see, Remarks~\ref{rmq::useful} and \ref{rmq::useful1}). Remember that $z(t,\cdot)\in \operatorname{D(M_{n})}$ implies $z(t,\cdot) = \sqrt{\cos x}v(t,\cdot)$ for some $v(t,\cdot)\in \operatorname{D}(\cL_{n})$. So, the above two inequalities lead to
	$$
		\int_{0}^{T}\int_{\operatorname{\omega}_{bdy}}\operatorname{K_{bdy}}dQ\ge \int_{0}^{T}\int_{\operatorname{\omega}_{bdy}}2s\theta|\dx z|^2+2s\theta|z|^2+\frac{\eta_1^2}{2}s^3\theta^3|z|^2dQ,
	$$
	completing the proof of lemma.
\end{proof}

The following lemma is a straightforward combination of Lemmas~\ref{Carleman.lem6} and \ref{Carleman.lem7}.
\begin{Lem}\label{Carleman.lem9}
	Let $n\in\NN$ and $\cR_{0}=\cR_{0}(a',b') := \max(s_1,s_2)$. Then for all
	\begin{equation}\label{eq98}
	s\ge \cR_{0}\max(T+T^2,T^2n),
	\end{equation}
	it holds
	\begin{eqnarray}\label{eq99}
	\int_{0}^{T}\int_{\operatorname{\omega}_{bdy}}\operatorname{K_{bdy}}dQ+\int_{0}^{T}\int_{\operatorname{\omega}_{deg}}\operatorname{K_{deg}}dQ&\ge&\int_{0}^{T}\int_{I\bs\operatorname{\omega}_{con}}(2s\theta|\dx z|^2+C_8s^3\theta^3|z|^2)dQ.
	\end{eqnarray}
	We let $C_8=C_8(a',b'):=\min(\eta_1^2,\eta_2^2)>0$ and $I\bs\operatorname{\omega}_{con}:=(-\pi/2,\pi/2)\bs\operatorname{\omega}_{con}=\operatorname{\omega}_{bdy}\cup\operatorname{\omega}_{deg}$. 
\end{Lem}
In the subcontrol region $\operatorname{\omega}_{con} = (-b',-a')\cup(a',b')$, we have the following
\begin{Lem}\label{Carleman.lem10}
	Let $n\in\NN^{*}$ and assuming \eqref{eq76} and \eqref{eq98}. Then there exist positive constants $C_9, C_{12}>0$ such that the following inequality holds
	\begin{equation}\label{eq100}
	|\operatorname{K_{con}}|\le C_9s\theta|\dx z|^2+C_{12}s^3\theta^3|z|^2.
	\end{equation}
\end{Lem}
\begin{proof}
	Let $n\in\NN^{*}$, then by Lemma \ref{Carleman.lem2}, \eqref{eq76} and \eqref{eq98} we have
	\begin{eqnarray*}
		|\operatorname{K_{con}}|&\le&\left|\frac{s\theta\beta^{(4)}}{2}+s\theta\beta'\left(2n^2-\frac{1}{2}\right)\frac{\sin x}{\cos^3x}-\frac {s\theta''\beta}2+2s^2\theta\beta'^2(\theta'-s\theta^2\beta'')\right||z|^2+|2s\theta\beta''||\dx z|^2\nonumber\\
		&\le&C_9s\theta|\dx z|^2+C_{11}s^3\theta^3|z|^2+\frac{C_{10}T^4}{8}n^2s\theta^3|z|^2,
	\end{eqnarray*}
	where $$C_9=C_9(\beta):=2\max\{|\beta''(x)|: x\in[-b',-a']\cup[a',b']\},$$
	$$C_{10} =C_{10}(\beta) := \max\left\{\left|\beta^{(4)}(x)\right|+\left|\frac{\sin x}{\cos^3x}\beta'(x)\right|: x\in[-b',-a']\cup[a',b']\right\},$$
	and
	$$C_{11}=C_{11}(\beta):=\frac{5\cR_{0}^{-2}\|\beta\|_\infty}{4}+2^{-5}\cR_{0}^{-2}C_{10}+(\cR_{0}^{-1}+C_9)\max\{|\beta'(x)|^2:x\in[-b',-a']\cup[a',b']\}.$$
	Finally, due to \eqref{eq98}, we complete the proof with  $C_{12}=C_{12}(a',b',\beta):=C_{10}+C_{11}/8\cR_{0}^2$.
\end{proof}
Thanks to \eqref{eq61}, \eqref{eq72}, Lemmas \ref{Carleman.lem9} and \ref{Carleman.lem10}, we immediately obtain the following.
\begin{Lem}\label{Carleman.lem11}
	Let $n\in\NN^{*}$. Then for all $s\ge \cR_{0}\max(T+T^2,T^2n)$, one has
	\begin{multline}\label{eq101}
	\int_{0}^{T}\int_{I\bs\operatorname{\omega}_{con}}(2s\theta|\dx z|^2+C_8s^3\theta^3|z|^2)dQ\le\\
	\int_{0}^{T}\int_{\operatorname{\omega}_{con}}(C_9s\theta|\dx z|^2+C_{12}s^3\theta^3|z|^2)dQ
	+\frac{1}{2}\int_{Q}\left|e^{-\varphi}\cP_ng\right|^{2}dQ.
	\end{multline}
\end{Lem}
In the following lemma, we come back to $g$.
\begin{Lem}\label{Carleman.lem12}
	Let $n\in\NN^{*}$ and assume \eqref{eq101}. Then there exist positive constants $C_{13}$, $C_{16}$ and $C_{17}$ such that  for all $s\ge \cR_{0}\max(T+T^2,T^2n)$, it holds
	\begin{multline}\label{eq102}
	\int_{Q}(C_{13}s\theta|\dx g|^2+(C_8/2)s^3\theta^3|g|^2)e^{-2\varphi}dQ\le\\
	\int_{0}^{T}\int_{\operatorname{\omega}_{con}}(C_{16}s\theta|\dx g|^2+C_{17}s^3\theta^3|g|^2)e^{-2\varphi}dQ
	+\frac{1}{2}\int_{Q}\left|e^{-\varphi}\cP_ng\right|^{2}dQ.
	\end{multline}
\end{Lem}
\begin{proof}
	By (\ref{eq53}), one has $\dx z=(\dx g-g\dx\varphi)e^{-\varphi}$. So, we obtain for all $\varepsilon>0$,
	$$-2g\dx g\dx\varphi \ge-|\partial_{x}g|^2/(1+\varepsilon)-(1+\varepsilon)|\partial_{x}\varphi|^2|g|^2,$$
	so that
	\begin{eqnarray*}
		2s\theta|\partial_{x}g-\partial_{x}\varphi g|^{2}+C_8s^3\theta^3|g|^2\ge 2s\theta\varepsilon|\partial_{x}g|^2/(1+\varepsilon)+s^3\theta^3(C_8-2\varepsilon|\beta'(x)|^2)|g|^2. 
	\end{eqnarray*}
	Letting now $\varepsilon=\varepsilon(\beta):=2^{-2}\|\beta'\|_{\infty}^{-2}C_8$,
	we have for all $s\ge\cC_{0}\max(T+T^2,T^2n),$
	\begin{multline}\label{eq103}
	\int_{0}^{T}\int_{I\bs\operatorname{\omega}_{con}}\hspace{-0.5cm}(C_{13}s\theta|\dx g|^2+(C_8/2)s^3\theta^3|g|^2)e^{-2\varphi}dQ\le\\
	\int_{0}^{T}\int_{\operatorname{\omega}_{con}}\hspace{-0.5cm}(C_{14}s\theta|\dx g|^2+C_{15}s^3\theta^3|g|^2)e^{-2\varphi}dQ
	+\frac{1}{2}\int_{Q}\left|e^{-\varphi}\cP_ng\right|^{2}dQ,
	\end{multline}
	where $C_{13}=C_{13}(\beta):=2\varepsilon/(1+\varepsilon)$, $C_{14}=C_{14}(\beta):=2C_9$, and $C_{15}=C_{15}(\beta) := C_{12}+2C_9\max\{|\beta'(x)|^2: x\in[-b',-a']\cup[a',b']\}$. By adding the same quantity $$\int_{0}^{T}\int_{\operatorname{\omega}_{con}}(C_{13}s\theta|\dx g|^2+(C_8/2)s^3\theta^3|g|^2)e^{-2\varphi}dQ$$ to the both sides of \eqref{eq103}, we complete the proof of lemma, with $C_{16}=C_{16}(\beta):=C_{13}+C_{14}$ and $C_{17}=C_{17}(\beta):=C_8/2+C_{15}$. 
\end{proof}
Let us prove that terms similar to the second term of the right-hand side of \eqref{eq102} dominate the first one. We achieve this by the use of a smooth cut-off function.
\begin{Lem}\label{Carleman.lem13}
	Let $n\in\NN^{*}$ and assume \eqref{eq102}. Then there exists a positive constant $C_{19}>0$ such that for all $s\ge \cR_{0}\max(T+T^2,T^2n)$, one has
	\begin{eqnarray}\label{eq104}
	\int_{Q}(C_{13}s\theta|\dx g|^2+(C_8/2)s^3\theta^3|g|^2)e^{-2\varphi}dQ&\le&\int_{0}^{T}\int_{\omega_{a,b}}C_{19}s^3\theta^3|g|^2e^{-2\varphi}dQ+\int_{Q}\left|e^{-\varphi}\cP_ng\right|^{2}dQ.\nonumber\\
	\end{eqnarray}
\end{Lem}
\begin{proof}
	Recall that $\operatorname{\omega}_{con}=(-b',-a')\cup(a',b')\subset\omega_{a,b}=(-b,-a)\cup(a,b)$ since $0<a<a'<b'<b\le\pi/2$. 
	Choosing a cut-off function $\rho\in\cC^{\infty}(\RR)$ such that $0\le\rho\le 1$ and
	\begin{equation}\label{eq105}
	\begin{array}{c c}
	\rho=1 &\hspace{0.3cm}\mbox{on }\operatorname{\omega}_{con}\\
	\rho=0 &\hspace{0.5cm}\mbox{on }I\bs\omega_{a,b},
	\end{array}
	\end{equation}
	we get
	\begin{eqnarray}\label{eq106}
	\int_{Q}(\cP_ng)g\rho\theta e^{-2\varphi}dQ
	&\ge&\int_{Q}\left[g\partial_{t}g-g\dx^2g-\frac{1}{2}|g|^2\right]\rho\theta e^{-2\varphi}dQ.
	\end{eqnarray}
	Or, one has
	$$\int_{Q}g\partial_{t}g\rho\theta e^{-2\varphi}dQ=\int_{Q}\frac{1}{2}|g|^2\rho(2\theta\partial_{t}\varphi+(T-2t)\theta^2)e^{-2\varphi}dQ,$$
	and
	$$\int_{Q}-g\dx^2g\rho\theta e^{-2\varphi}dQ=\int_{Q}\rho e^{-2\varphi}\theta|\dx g|^2dQ-\int_{Q}\frac{|g|^2}{2}e^{-2\varphi}\theta(\rho''-4\rho'\dx\varphi+\rho(4|\dx\varphi|^2-2\dx^2\varphi))dQ.$$
	Thus, from (\ref{eq106}), we obtain for all $s\ge\cR_{0}\max(T+T^2,T^2n),$
	\begin{eqnarray}\label{eq107}
	\int_{0}^{T}\int_{\operatorname{\omega}_{con}}C_{16}s\theta|\dx g|^2e^{-2\varphi}dQ&\le&\int_{Q}C_{16}s \theta\rho|\dx g|^2e^{-2\varphi}dQ\nonumber\\
	&\le&\frac{1}{2}\int_{Q}|e^{-\varphi}\cP_ng|^2dQ+\int_{0}^{T}\int_{\omega_{a,b}}C_{18}s^3\theta^3|g|^2e^{-2\varphi}dQ,
	\end{eqnarray}
	where the positive constant, $$
	C_{18}=C_{18}(\beta,\rho) :=\frac{C_{16}^2}{8\cR_{0}^2}+ C_{16}\left[4\|\beta'\|_{\infty}^2+\frac{\|\rho'\beta'\|_{\infty}+\|\beta''\|_{\infty}+4\|\beta\|_{\infty}}{2\cR_{0}}+\frac{9+\|\rho''\|_{\infty}}{32\cR_{0}^{2}}\right].$$ 
	So we deduce from \eqref{eq102} and \eqref{eq107} that, for all $s\ge \cR_{0}\max(T+T^2,T^2n)$, one has
	$$	\int_{Q}(C_{13}s\theta|\dx g|^2+(C_8/2)s^3\theta^3|g|^2)e^{-2\varphi}dQ\le\int_{0}^{T}\int_{\omega_{a,b}}C_{19}s^3\theta^3|g|^2e^{-2\varphi}dQ+\int_{Q}\left|e^{-\varphi}\cP_ng\right|^{2}dQ,$$
	and this completes the proof of lemma with $C_{19} = C_{19}(\beta,\rho):=C_{18}+C_{17}$.
\end{proof}
We can complete now the proof of Carleman estimate \eqref{eq48}.
\begin{proof}[Proof of Proposition \ref{Carleman.pro}]
	It suffices to consider Lemma \ref{Carleman.lem13}, and let 
	$$\cR_{1}=\cR_{1}(\beta,\rho):=\frac{\min(C_{13},C_8/2)}{\max(1,C_{19})}. \qedhere$$
\end{proof}
\section{Proof of Theorem \ref{Theo2}}\label{s.proof}
In this section we present the proof of Theorem \ref{Theo2}, along the following lines:
\begin{enumerate}
	\item The proof of the negative statement, presented in Section \ref{ss.proof_negative}, relies on the use of appropriate test function (which concentrate at zero when $n=l\rightarrow+\infty$) to falsify uniform observability inequality \eqref{eq46} when $T\le\log(1/\cos a)$;
	\item The proof of the positive statement, presented in Section \ref{ss.proof_positive}, relies on the uniform observability inequality \eqref{eq46} in large time, by using the global Carleman estimate for system \eqref{eq23}, proved in the previous section.
\end{enumerate}

\subsection{Proof of the negative statement of Theorem \ref{Theo2}}\label{ss.proof_negative}
The goal of this subsection is to prove
\begin{Prop}\label{Negative.pro}
	Let $a,b\in\RR$ be such that $0<a<b\le\pi/2$ and $T\le\log\left(1/\cos a\right)$. Then system \eqref{eq23} is not observable in $(a,b)$ in time $T$ uniformly with respect to $n\in\NN^{*}$.
\end{Prop}
\begin{Rema}
	Note that, the not null observability result provided here remains true in $(-b,-a)\cup(a,b)$ by symmetry and parity. However, this concerns only the cases $n\in\NN^{*}$ (see, Remark \ref{rmk::usefule}).
\end{Rema}
\begin{proof}
	We use a particular function which solves \eqref{eq23} and for which the observability inequality \eqref{eq46} fails under the condition $T\le\log\left(1/\cos a\right)$.
	More precisely, in what follows, we design a sequence of solutions of \eqref{eq23} such that
	\begin{equation}\label{eq108}
	\frac{\displaystyle\int_{0}^{T}\int_a^b|g_n(t,x)|^2\cos xdxdt}{\displaystyle\int_{-\frac\pi 2}^{\frac\pi 2}|g_n(T,x)|^2\cos xdx}\longrightarrow 0\hspace{0.5cm}\mbox{as}\hspace{0.5cm}n\longrightarrow+\infty.
	\end{equation}
	
	We recall that the highest weight spherical harmonics of degree $n$ present extreme concentration around the equator. These are defined by
	\begin{equation}\label{eq109}
	W_{n,n}(x,y) = \frac{(-1)^n}{2^nn!}\sqrt{\frac{(2n+1)!}{4\pi}}e^{iny}\cos^{n}x,
	\end{equation}
	where $n\in\NN^{*}$, $(x,y)\in[-\pi/2,\pi/2]\times[0,2\pi)$. Consider the function
	\begin{equation}\label{eq110}
	w_{n}(x) = \frac{(-1)^n}{2^nn!}\sqrt{(2n+1)!}(\cos x)^{n},\hspace{0.3cm}n\in\NN^{*},\;x\in[-\pi/2,\pi/2].
	\end{equation}
	By Wallis' formula, we have for every $n\in\NN^{*}$,
	$\int_{0}^{\frac\pi 2}\cos^{2n+1}xdx = 2^{2n}(n!)^2/(2n+1)!$.
	So for every $n\in\NN^{*}$, we deduce 
	$$\int_{-\frac\pi 2}^{\frac\pi 2}w_{n}^{2}(x)\cos xdx = \frac{(2n+1)!}{2^{2n}(n!)^{2}}\int_{0}^{\frac\pi 2}\cos^{2n+1}xdx = 1.$$
	
	We check easily now that for every $n\in\NN^{*}$, the function
	\begin{equation}\label{eq111}
	g_n(t,x) = e^{-nt}w_{n}(x),\hspace{0.5cm}t\in\RR,\;x\in[-\pi/2,\pi/2],
	\end{equation}
	solves the system \eqref{eq23} and
	$\int_{-\frac\pi 2}^{\frac\pi 2}|g_n(T,x)|^{2}\cos xdx = e^{-2nT}$.
	To get \eqref{eq108}, it suffices to prove that
	\begin{equation}\label{eq112}
	\frac{e^{2nT}}{2n}\int_a^bw_{n}(x)^{2}dx\rightarrow 0\hspace{0.3cm}\mbox{as}\hspace{0.3cm}n\rightarrow+\infty.
	\end{equation}
	Since $0<a<b\le \pi/2$, we have $\displaystyle\int_{a}^{b}(\cos x/\cos a)^{2n+1}dx\le(b-a)$. Then, we obtain
	\begin{eqnarray}\label{eq113}
	\frac{e^{2nT}}{2n}\int_a^bw_{n}(x)^{2}\cos xdx
	&\le&e^{2n(T+\ln\cos a)}\frac{(2n+1)!}{n2^{2n+1}(n!)^{2}}(b-a)\cos a.
	\end{eqnarray}
	By Stirling's formula, we have $n!\sim\sqrt{2\pi n}n^{n}e^{-n}$ as $n\rightarrow+\infty$. Thus, we deduce from \eqref{eq113} that $$\frac{e^{2nT}}{2n}\int_{a}^{b}w_{n}(x)^{2}\cos xdx\le(b-a)\frac{\cos a}{2\sqrt{\pi}}\frac{(2n+1)}{n^{3/2}}\longrightarrow 0\hspace{0.3cm}\mbox{as}\hspace{0.3cm}n\longrightarrow+\infty,$$
	since we are assuming $T\le\log\left(1/\cos a\right)$. This completes the proof of the proposition. 
\end{proof}

\subsection{Proof of the positive statement of Theorem \ref{Theo2}}\label{ss.proof_positive}

This subsection is devoted to the prove of following proposition using the Carleman estimate \eqref{eq48} and dissipation rate \eqref{eq32},
\begin{Prop}\label{Positive.pro}
	Let $a,b\in\RR$ be such that $0<a<b\le\pi/2$. Then there exists a positive time $T^{*}>0$ such that, for every $T\ge T^{*}$, system \eqref{eq23} is observable in $\omega_{a,b}=(-b,-a)\cup(a,b)$ in time $T$ uniformly with respect to $n\in\NN$.
\end{Prop}
\begin{proof}
	We obtain the uniform observability inequality \eqref{eq46} in large time from observability inequality \eqref{eq::inequality 0} and Carleman estimate \eqref{eq48}. Let $n\in\NN^{*}$ and $\tilde{g}_n = \operatorname{U}g_n\in \cC([0,T]; L^2(-\pi/2,\pi/2))\cap\cC^{2}((0,T);\operatorname{D(M_{n})})$ be the solution of system \eqref{eq39}, where $g_n$ is the Fourier component \eqref{eq24} and $\operatorname{U}$, the unitary transformation introduced in Section \ref{ss.Well3}. Then by the Carleman estimate \eqref{eq48}, one has
	\begin{eqnarray}\label{eq114}
	\cR_{1}	\int_{Q}\theta^3|\tilde{g}_n(t,x)|^2e^{-2\varphi}dQ&\le&\int_{0}^{T}\int_{\omega_{a,b}}\theta^3|\tilde{g}_n(t,x)|^2e^{-2\varphi}dQ,
	\end{eqnarray}
	for all $s\ge \cR_{0}\max(T+T^{2},T^{2}n)$, and for some constants $\cR_{0},\cR_{1}>0$ independent of $n$, $T$ and $\tilde{g}_n$. From now on we set
	$$s:= \cR_{0}\max(T+T^{2},T^{2}n).$$
	For $t\in(T/3,2T/3)$, we have owing to dissipation rate (\ref{eq39})
	$$\frac{4}{T^2}\le\theta(t)\le\frac{9}{2T^2}\hspace{0.5cm}\mbox{and}\hspace{0.5cm}\int_{-\frac\pi 2}^{\frac\pi 2}|\tilde{g}_n(T,x)|^2dx\le e^{-\frac{2}{3}nT}\int_{-\frac\pi 2}^{\frac\pi 2}|\tilde{g}_n(t,x)|^2dx.$$
	Integrating over $(T/3,2T/3)$, we have using \eqref{eq114}
	\begin{equation}\label{eq115}
	\frac{T}{3}\int_{-\frac\pi 2}^{\frac\pi 2}|\tilde{g}_n(T,x)|^2dx\le
	\frac{1}{\cR_{1}}\frac{T^6}{64}\frac{6}{8s^3\beta_{*}^3}e^{-\frac{2}{3}nT}e^{\frac{9}{T^2}s\beta^{*}}\int_{0}^{T}\int_{\omega_{a,b}}|\tilde{g}_n(t,x)|^2dxdt,
	\end{equation}
	where $\beta_{*}:=\min\{\beta(x):x\in[-\pi/2,\pi/2]\}$ and $\beta^{*}:=\max\{\beta(x):x\in[-\pi/2,\pi/2]\}$. Then, the following two cases may occur
	\begin{enumerate}
		\item[]\textbf{First case:} $n<1+1/T$. Then $s = \cR_{0}(T+T^2)$, and thus \eqref{eq115} yields 
		$$\int_{-\frac\pi 2}^{\frac\pi 2}|\tilde{g}_n(T,x)|^2dx
		\le\frac{1}{\cR_{1}}\frac{T^5}{64}\frac{18}{8\cR_{0}^3(T+T^2)^3\beta_{*}^3}e^{\frac{9}{T^2}\cR_{0}(T+T^2)\beta^{*}}\int_{0}^{T}\int_{\omega_{a,b}}|\tilde{g}_n(t,x)|^2dxdt;$$ 
		\item[]\textbf{Second case:} $n\ge1+1/T$. Then $s = \cR_{0}T^2n$, and thus \eqref{eq115} yields 
		$$\int_{-\frac\pi 2}^{\frac\pi 2}|\tilde{g}_n(T,x)|^2dx\le\frac{1}{\cR_{1}}\frac{1}{64}\frac{18}{8\cR_{0}^3T(1+1/T)^3\beta_{*}^3}e^{-\frac{2}{3}nT}e^{9n\cR_{0}\beta^{*}}\int_{0}^{T}\int_{\omega_{a,b}}|\tilde{g}_n(t,x)|^2dxdt.$$
	\end{enumerate}
It then suffices to observe that $\displaystyle-2nT/3+9n\cR_{0}\beta^{*}\le 0$ as soon as $\displaystyle T\ge T^{*}:=27\cR_{0}\beta_{*}/2$.

So, in both cases, there exists a positive constant $C_0'>0$ which is independent of $n\in\NN^{*}$, such that
$$
\int_{-\frac\pi 2}^{\frac\pi 2}|g_n(T,x)|^2\cos xdx\le C_0'\int_{0}^{T}\int_{\omega_{a,b}}|g_n(t,x)|^2\cos xdxdt,
$$
provided $T\ge T^{*}$. Then the above identity and \eqref{eq::inequality 0} assure existence of a positive constant $C:=\max(C_0,C_0')>0$ independent on $n\in\NN$ such that \eqref{eq46} holds true, provided $T\ge T^{*}$. This completes the proof of the proposition.

\end{proof}

%
%


\end{document}